\documentclass[11pt]{article}
\usepackage[usenames]{color}
\usepackage{amsthm}
\usepackage{amsmath}
\usepackage{amssymb}
\usepackage{amsfonts}
\usepackage{enumerate}
\usepackage[tmargin=1.1in,bmargin=1.0in,lmargin=1.0in,rmargin=1.4in]{geometry}
\usepackage{tikz}
\usepackage{hyperref}
\begin{document}
  \newcommand{\dcup}{\makebox[0pt]{\quad \! \cdot} \cup}
  \newcommand{\om}{m^{*}}
  \newcommand{\eps}{\epsilon}
  \newcommand{\ea}{\begin{eqnarray*}}
  \newcommand{\ee}{\end{eqnarray*}}
  \newcommand{\nea}{\begin{eqnarray}}
  \newcommand{\nee}{\end{eqnarray}}
  \newcommand{\RRR}{\mathfrak{R}}
  \newcommand{\NNN}{\mathfrak{N}}
  \newcommand{\FFF}{\mathfrak{F}}
  \newcommand{\SSS}{\mathfrak{S}}
  \newcommand{\mmm}{\mathfrak{m}}
  \newcommand{\aaa}{\mathfrak{a}}
  \newcommand{\bbb}{\mathfrak{b}}
  \newcommand{\ccc}{\mathfrak{c}}
  \newcommand{\ppp}{\mathfrak{p}}
  \newcommand{\qqq}{\mathfrak{q}}
  \newcommand{\Aa}{\mathcal{A}} \newcommand{\Bb}{\mathcal{B}}
  \newcommand{\Mm}{\mathcal{M}}
  \newcommand{\Rr}{\mathcal{R}}
  \newcommand{\Ee}{\mathcal{E}}
  \newcommand{\Ff}{\mathcal{F}}
  \newcommand{\Cc}{\mathcal{C}}
  \newcommand{\Tt}{\mathcal{T}}
  \newcommand{\Nn}{\mathcal{N}}
  \newcommand{\RR}{\mathbb{R}}
  \newcommand{\CC}{\mathbb{C}}
  \newcommand{\eRR}{\overline{\RR}}
  \newcommand{\QQ}{\mathbb{Q}}
  \newcommand{\NN}{\mathbb{N}}
  \newcommand{\PP}{\mathbb{P}}
  \newcommand{\ZZ}{\mathbb{Z}}
  \renewcommand{\AA}{\mathbb{A}}
  \newcommand{\limn}{\lim_{n\to\infty}}
  \newcommand{\liminfn}{\liminf_{n\to\infty}}
  \newcommand{\limsupn}{\limsup_{n\to\infty}}
  \newcommand{\liminfj}{\liminf_{j\to\infty}}
  \newcommand{\limsupj}{\limsup_{j\to\infty}}
  \newcommand{\limj}{\lim_{j\to\infty}}
  \newcommand{\limi}{\lim_{i\to\infty}}
  \newcommand{\sumj}{\sum_{j=1}^{\infty}}
  \newcommand{\sumi}{\sum_{i=1}^{\infty}}
  \newcommand{\sumn}{\sum_{n=1}^{\infty}}
  \newcommand{\Lra}{\Longrightarrow}
  \newcommand{\Lla}{\Longleftarrow}
  \newcommand{\floor}[1]{\left\lfloor #1 \right\rfloor}
  \newcommand{\ceil}[1]{\left\lceil #1 \right\rceil}
  \newcommand{\norm}[1]{\|#1\|_u}
  \newcommand{\fh}{\hat{f}}
  \newcommand{\Ft}{\widetilde{F}}
  \newcommand{\ft}{\tilde{f}}
  \newcommand{\gh}{\hat{g}}
  \newcommand{\ffint}{\frac{1}{2\pi}\int_{-\pi}^\pi}
  \newcommand{\fint}{\int_{-\pi}^\pi}
  \newcommand{\fp}[1]{\langle #1 \rangle}
  \newcommand{\net}[1]{\langle #1 \rangle}
  \renewcommand{\Re}{\mathrm{Re}}
  \newcommand{\var}{\mathop{\mathrm{var}}\nolimits}
  \renewcommand{\subset}{\subseteq}
  \renewcommand{\supset}{\supseteq}
  \renewcommand{\emptyset}{\varnothing}
  \newcommand{\diam}{\mathop{\mathrm{diam}}}
  \newcommand{\krull}{\mathop{\mathrm{Krull\ dim}}}
  \newcommand{\card}{\mathop{\mathrm{card}}}
  \newcommand{\spec}{\mathop{\mathrm{Spec}}}
  \newcommand{\rad}{\mathop{\mathrm{rad}}}
  \newcommand{\tr}{\mathop{\mathrm{tr}}}
  \newcommand{\im}{\mathop{\mathrm{Im}}}
  \newcommand{\tor}{\mathop{\mathrm{Tor}}\nolimits}
  \newcommand{\ext}{\mathop{\mathrm{Ext}}\nolimits}
  \renewcommand{\hom}{\mathop{\mathrm{Hom}}\nolimits}
  \newcommand{\id}{\mathop{\mathrm{Id}}\nolimits}
  \renewcommand{\ker}{\mathop{\mathrm{Ker}}\nolimits}
  \newcommand{\coker}{\mathop{\mathrm{Coker}}\nolimits}
  \newcommand{\ann}{\mathop{\mathrm{Ann}}}
  \newcommand{\supp}{\mathop{\mathrm{Supp}}}
  \newcommand{\ems}{m^*}
  \newcommand{\znq}{{\ZZ^{\NN}_q}}
  \newcommand{\zq}[1]{\ZZ^{#1}_q}
  \newcommand{\cl}[1]{\overline{#1}}
  \newcommand{\ah}{\hat{a}}
  \newcommand{\bh}{\hat{b}}
  \newcommand{\ms}{\mu^{*}}
  \newcommand{\Hh}{\mathcal{H}}
  \newcommand{\Uj}{\bigcup_{j=1}^{\infty}}
  \newcommand{\Ui}{\bigcup_{i=1}^{\infty}}
  \newcommand{\Un}{\bigcup_{n=1}^{\infty}}
  \newcommand{\Ij}{\bigcap_{j=1}^{\infty}}
  \newcommand{\In}{\bigcap_{n=1}^{\infty}}
  \newcommand{\E}{\mathbf{E}}
  \renewcommand{\P}{\mathbf{P}}
  \renewcommand{\var}{\mathbf{Var}}
  \newtheorem{thm}{Theorem}
  \newtheorem*{restate}{Theorem 3}
  \newtheorem{prop}[thm]{Proposition}
  \newtheorem{lemma}[thm]{Lemma}
  \newtheorem{cor}[thm]{Corollary}
  \theoremstyle{definition}
  \newtheorem{defn}[thm]{Definition}
  \theoremstyle{plain}

\title{The Mar\v{c}enko-Pastur law for sparse random\\ bipartite biregular
 graphs}
 \date{September 9, 2014}
\author{Ioana Dumitriu\thanks{Both authors acknowledge support from NSF 
 CAREER Award DMS-0847661. The second author also received support from NSF grant
 DMS-1401479.}\\
 \texttt{dumitriu@math.washington.edu}\\
 \hfill\\
 Department of Mathematics\\ University of Washington\\Box 354350\\ Seattle, WA 98195
 \and
 Tobias Johnson\\
 \texttt{toby@math.washington.edu}\\\hfill\\
 Department of Mathematics\\
 University of Southern California\\
 3620 S.~Vermont St, KAP~108\\
 Los Angeles, CA 90089
 }
\maketitle
  \begin{abstract}
We prove that the empirical spectral distribution of a $(d_L, d_R)$-biregular, 
bipartite random graph, under certain conditions, converges
to a symmetrization of the Mar\v{c}enko-Pastur distribution of random
matrix theory. This convergence is not only global (on fixed-length
intervals) but also local (on intervals of increasingly smaller length).  Our method parallels the one used previously by Dumitriu and Pal (2012).  
  \end{abstract}

  \section{Motivation}\label{sec:intro}

In classical random matrix theory there are two basic types of symmetric
ensembles: Wigner matrices and Wishart-like ones.
There is a simple parallel between them,
roughly expressed in the following way. Given a non-symmetric real
matrix $G$, there are two natural ways to construct from it a symmetric
matrix: if $G$ is square, one way is to consider its symmetric part
$A = \frac{G+G^{T}}{2}$; the other way works for rectangular matrices,
too, and consists of multiplying it by its transpose: $W = GG^{T}$. If
one starts with a random matrix $G$ with i.i.d.\ entries of norm~$0$
and variance~$1$, if $G$ is square, the first symmetrization yields a
Wigner matrix; the second symmetrization yields a Wishart-like matrix
(it is Wishart, more precisely central Wishart, if $G$ consists of
standard normal variables; we call it Wishart-like otherwise). 

The spectra and eigenvectors of Wigner and Wishart-like matrices have been
shown to exhibit \emph{universal} behavior: many of their
eigenstatistics have limiting distributions which are independent of
the \emph{entry} distribution, modulo certain technical conditions (in
addition to being mean 0, variance 1). These results, along with
successive weakenings of the technical conditions, are the subject of
a recent set of breakthrough papers \cite{TaV,TV2,ESY1,ESY2,ESY3,ESY4,ESY5}.

A natural question for the discrete probability community is whether this universal behavior extends to adjacency matrices of random graphs; we are specifically interested in the case of random regular or semi-regular graphs. Such graphs are known to have very interesting properties: they are good expanders, some classes are quasi-Ramanujan, they have wide spectral gaps and as such they mix rapidly, and they are of interest in computer science and engineering and in coding theory.

Ordinary random regular graphs (or $d$-regular graphs), where every vertex has the same degree (which grows as a function of the number of vertices), have been recently investigated in \cite{DuP, DJPP, J, JP, TVW, BL}; we aim to extend some of the results to bipartite, biregular random graphs, where the two sets of vertices have the property that all vertices in the same set have the same degree (also growing with the total number of vertices). 

The question of whether the spectra of random $d$-regular graphs have the
same behavior as the spectra of Wigner matrices is non-trivial in
nature, since the adjacency matrices of regular graphs have strong
dependencies, namely, all rows and columns add to the same
number---their common degree. As such, they are not Wigner; in fact,
for $d$ fixed, McKay \cite{McK} showed that the scaled \emph{empirical
  distribution  function} (or ESD; defined below) 
of random $d$-regular graphs on $n$ vertices converges in probability 
(and almost surely,
if the random graphs are defined on the same probability space) 
as $n\to\infty$ to
  the Kesten-McKay distribution, which has density
\[
f_d(x)=\begin{cases}
               \frac{d\sqrt{4(d-1)-x^2}}{2\pi(d^2-x^2)}&\text{if $|x|\leq 
                 2\sqrt{d-1}$,}\\0&\text{otherwise.}
              \end{cases}
\]
This differs from the Wigner matrix case, where the scaled ESD converges in probability to the semicircle law, which has density
\[
f_s(x) =\begin{cases}
  \frac{\sqrt{4-x^2}}{2\pi}&\text{if $|x|\leq 2$,}\\
  0 &\text{otherwise.}
\end{cases}
\]

When $d$ is allowed to grow with $n$, the scaled empirical spectral
distribution of the random $d$-regular graph does converge in
probability to the semicircle law  (see \cite{DuP} and \cite{TVW}).  Thus, even if for $d$ fixed they are rather different, the spectra of $d$-regular graphs with $d$ and $n$  growing to infinity are similar to the spectra of large Wigner matrices. 

Motivated by a question asked by Babak Hassibi, we study here the
spectra of bipartite, biregular random graphs. We find that their
spectra have similar behavior to the spectra of Wishart-like matrices;
at first glance, this may appear suprising, but further examination
reveals linear algebraic reasons why this should be so.
Our notation and main results are presented in Section~\ref{prelim}.
In Sections~\ref{sec:main} and \ref{sec:locallaw}, we prove global and local
convergence, respectively, of the ESD to its limiting measure.
The appendix contains the proofs of some statements on the distribution
of cycles in biregular bipartite graphs, used in
Sections~\ref{sec:main} and \ref{sec:locallaw} to show that our graphs
are locally well approximated by trees.

%

\section{Preliminaries and statements of results} \label{prelim}
We assume that graphs do not have loops or parallel edges.
The adjacency matrix of a graph $G$ is defined as 
\[
A(i,j) = \left \{ \begin{array}{cr}1, &~~~\mbox{if}~ i \sim j, \\
                                                 0,
                                                 &~~~\mbox{otherwise.}
\end{array} \right .
\]
Note that $A$ is symmetric, and therefore all of its eigenvalues are
real.

Bipartite graphs are graphs composed of two sets $L$ and
$R$ of vertices, with edges only between vertices from $L$ and
vertices from $R$. With proper labeling, their adjacency matrices 
have the special form 
\[
B = \left [  \begin{array}{cc} 0 & X \\ X^{T} & 0 \end{array} \right ]~,
\]
where the matrix $X$ is defined by the edges between the two classes of
vertices. Note that if $L$ and $R$ have sizes $m$ and
 $n$, respectively, with $m \leq n$, then $X$ is an $m
\times n$ matrix of $0$s and $1$s.
It is a simple linear algebra result that the non-zero eigenvalues of
$B$ come in pairs $(-\lambda, \lambda)$, with $\lambda \geq 0$ an
eigenvalue of $XX^{T}$, and that $B$ has (at least) $n-m$ eigenvalues
equal to $0$.

If $G$ is a bipartite graph with vertex classes $L$ and $R$, then we say it is 
  $(d_L, d_R)$-biregular
  if all the vertices in $L$ have degree $d_L$, and all the vertices in $R$ have
  degree $d_R$.  For simplicity, we will always assume
  that $d_R\geq d_L$ (and therefore that $|L|\geq |R|$).  
  
By a random $(d_L, d_R)$-biregular bipartite graph with
  $(m+n)$ vertices we mean a graph selected uniformly from the space
  of all $(d_L,d_R)$-biregular bipartite graphs with $|L|=m$
  and $|R|=n$.  

We define the \emph{empirical spectral distribution} or
  ESD of a symmetric $n\times n$ matrix $M$ 
  to be the probability
  measure $\mu_n$ on the real numbers given by
  \begin{align*}
    \mu_n = \frac{1}{n}\sum_{i=1}\delta_{\lambda_i},
  \end{align*}
  where $\delta_x$ is the point mass at $x$ and $\lambda_1,\ldots,
  \lambda_n$ are the eigenvalues of $M$.  Note that if $M$ is random, $\mu$ is a
  random probability measure. 

We say that a sequence  $\mu_1,\mu_2,\ldots$
  of random probability measures on the real numbers converges almost surely
  to a deterministic probability measure $\mu$ if as $n\to\infty$,
  \begin{align*}
    \int f\, d\mu_n \to \int f\,d\mu \text{ a.s.}
  \end{align*}
  for all bounded continuous functions $f\colon\RR\to\RR$. This is
  equivalent to the slightly different statement that with probability
  one, $\mu_n$ converges weakly to $\mu$.
  
Following combinatorialists'
  conventions, we will often refer to a random graph $G$ when we really
  mean a sequence of random graphs with an increasing
  number of vertices (depending on $m$ and $n$).
  We assume that all of these random graphs are defined on
  the same probability space, but we make no assumptions 
  about their joint distribution;
  all of our results hold for any arbitrary joint distribution, so long
  as the marginal distributions are as described.
  Many
  variables that we mention implicitly depend on $n$, and asymptotic
  expressions $O(\cdot)$ or $o(\cdot)$ reflect
  behavior as $m, ~n\to\infty$.  We will occasionally use the
  notation $O_A(\cdot)$ to indicate that the constant
  in the big-O expression depends on some other constant $A$.
  
  
  It is a standard result in random matrix theory
  that if $X$ is an $m\times n$ random matrix whose entries are 
  i.i.d.\ real random variables with mean zero and variance one,
  and $m/n$ converges to a finite limit,
  then the ESD of $\frac{1}{n}X^TX$ converges to
  the Mar\v{c}enko-Pastur law (see \cite{Bai}, for example).
  We show here an analogous result:
  \begin{thm}\label{thm:sq}
    Let $G$ be a random $(d_L,d_R)$-biregular bipartite graph
    on $m+n$ vertices, with the following conditions
    on the growth of $d_L$ and $d_R$:
    \begin{align}
      \limn d_R&=\infty,\label{eq:dc1}\\
      \mbox{for any fixed $\epsilon>0$,}~~d_R &= o( n^{\epsilon}),\label{eq:dc2}\\
      \frac{d_R}{d_L} &\to y\geq 1.\label{eq:dc3}
    \end{align}
    Let $A = \left(\begin{smallmatrix} 0&X\\X^T & 0\end{smallmatrix}
    \right)$ be the adjacency matrix of $G$ (under proper labeling).  Then as $n\to\infty$,
    the ESD
    of $\frac{1}{d_R} X^TX$ converges almost surely to the
    Mar\v{c}enko-Pastur law with ratio $y^{-1}$.
    This distribution
    is supported on $[a^2,b^2]$ and is given on that
    interval by the density
    \begin{align*}
      p(x)=\frac{y}{2\pi x}\sqrt{(b^2-x)(x-a^2)},
    \end{align*}
    where $a = 1-y^{-1/2}$ and $b = 1+y^{-1/2}$.
  \end{thm}

  Theorem \ref{thm:sq} agrees with the results of \cite{MS}, in which Mizuno and Sato derive the 
  limiting distribution of the eigenvalues of a sequence of deterministic biregular graphs with girths growing to
  infinity. They do so using the Ihara zeta function, and they express
  their result  in a different form from ours, which 
  turns out to be equivalent.  
It should be noted that our result is much stronger
  than theirs; as we will see, \emph{most} bipartite biregular
  graphs have small girth, but this does not affect the
  convergence of the ESD.

  As mentioned above, the ESD of $d_R^{-1}X^TX$ is the distribution
  of the squares of the nontrivial eigenvalues of $d_R^{-1/2}A$.
  We can thus find the limiting distribution for the ESD
  of this matrix as well:
  \begin{cor}\label{cor:mudist}
    The ESD of $d_R^{-1/2}A$ converges almost surely to the distribution 
    $\mu$ supported on $[-b,-a]\cup[a,b]$ and given on that set by
    the density
    \begin{align}
      \frac{2}{1+y}p(x^2)|x|=\frac{y}{(1+y)\pi|x|}
      \sqrt{(b^2-x^2)(x^2-a^2)},\label{eq:density}
    \end{align}
    along with a point mass of $\frac{y-1}{y+1}$ at $0$.
  \end{cor}

  It is also known that when $d\to\infty$, the ESD of random $d$-regular
  graphs converges to the semicircle law on short scales (see
  \cite{DuP}, \cite{TVW}).  In Section 
  \ref{sec:locallaw}, we prove this
  for the biregular case, under slightly different
  conditions on the growth of $d_R$:
  \begin{thm}\label{thm:locallaw}
    Let $G$ be a random $(d_L,d_R)$-biregular bipartite graph on
    $m+n$ vertices satisfying \eqref{eq:dc1}--\eqref{eq:dc3},
    as well as the more stringent condition
    $d_R=\exp(o(1)\sqrt{\log n})$.
    Fix $\eps>0$.  Let $A$ be the adjacency matrix
    of $G$ and $\mu_n$ be the ESD
    of $(d_R-1)^{-1/2}A$, and let $\mu$ be the limiting ESD of
    Corollary~\ref{cor:mudist}.
    There exists a constant $C_\eps$ such that
    for all sufficiently large $n$ and $\delta>0$,
    for any interval $I\subset\RR$ 
    avoiding
    $[-\eps,\eps]$ and with length $|I|\geq\max\big(2\eta,\eta/(-\delta
    \log\delta)\big)$, it holds that
    \begin{align*}
      |\mu_n(I)-\mu(I)|<\delta C_\eps|I|
    \end{align*}
    with probability $1-o(1/n)$.  The quantity $\eta$, which gives the
    minimum length of an interval $I$ that we consider, is given by
    the following series of definitions:
    \begin{align*}
      a &= \min\left(\frac{\log n}{9(\log d_R)^2}, d_R\right),\\
      r &= e^{1/a},\\
      \eta &= r^{1/2}-r^{-1/2}.
    \end{align*}
  \end{thm}

  This theorem is obscured by the technicalities in its statement,
  so we give some discussion of its meaning.  
  Corollary~\ref{cor:mudist} only gives information on $\mu_n(I)$ for fixed
  $|I|$.  Theorem~\ref{thm:locallaw}, on the other hand, allows $|I|$ to shrink
  as $\eta$ does.  We note that $\eta\sim 1/a$.
  
  We restricted our intervals $I$ away from 0 to avoid complications
  caused by the point mass that $\mu$ has when $d_R/d_L\to
  y>1$.  Since the support of $\mu$ except for this mass is bounded
  away from 0, this restriction costs us nothing.

 To prove our results, we use the moment method along with Stieltjes transforms, combined with a careful examination of the local structure of the bipartite, biregular graph. 
Along the way, we need to adapt some of the results proved by McKay, Wormald, and Wysocka \cite{MWW} for random $d$-regular graphs to our $(d_L, d_R)$ biregular random graphs. We give these results below. The method used follows \cite{MWW} very closely, which is why we relegate the proofs to the appendix. 

Let $G$ be a random $(d_L,d_R)$-biregular bipartite graph on $m+n$
  vertices.  Assume  $d_L\leq d_R$, and let $\alpha=d_R/d_L$.  As always, all of these
  variables depend on $n$, and any expressions $O(\cdot)$ or $o(\cdot)$ reflect
  behavior as $n\to\infty$.  We assume that $\alpha$ converges to a finite
  value as $n\to\infty$ (this assumption is also necessary in the case
  of Wishart matrices).  Here and throughout the paper, ``cycle'' always refers to a simple
  cycle, with no repeated vertices.
  Let $X_r$ denote the number of cycles of length $2r$ in $G$. (Note that
  as $G$ is bipartite, its cycles all have even length.)
  \begin{prop}\label{prop:ev}
    Let
    \begin{align*}
      \mu_r = \frac{(d_L-1)^r(d_R-1)^r}{2r}.
    \end{align*}
    If $d_R=o(n)$, $r=O(\log n)$, and $rd_R = o(n)$,
    then
    \begin{align*}
      \E[X_r] &=
        \mu_r
        \left(1 + O\Big(\frac{r(r+d_R)}{n}\Big)\right),\\
      \var[X_r] &=    
      \mu_r\left(1+O\bigg(\frac{d_R^{2r}(r\alpha^{2r-1}+\alpha^{-r}d_R)}
          {n}\bigg)\right)
    \end{align*}
  \end{prop}
  
  In Section~\ref{sec:main}, we use this proposition
  to show that with high probability, our biregular bipartite graph
  $G$ is locally well approximated by a tree.  This allows us
  to approximate the traces of the adjacency matrix of $G$, thus
  computing the moments of its ESD.
  Finally, we refine these results
  in Section~\ref{sec:locallaw} to prove local convergence on vanishing-length
  intervals.
  To prove this theorem, we give estimates on the rate of convergence
  of the Stieltjes transform of the ESD, the same approach used in
  \cite{DuP} and \cite{TaV}.

  \section{Global convergence to the Mar\v{c}enko-Pastur law}\label{sec:main}
  To find the limiting ESD of a biregular bipartite graph $G$, we will
  first show that in a sense that we will make precise, 
  most neighborhoods in $G$ have no cycles and are trees.
  This will allow us to estimate the traces of the adjacency matrix of $G$,
  and we will find the limit of these as $G$ grows 
  with a combinatorial argument.
  
  For this entire section, let $G$ be a random $(d_L,d_R)$-biregular
  bipartite graph on $n+m$ vertices, and assume
  that conditions \eqref{eq:dc1}--\eqref{eq:dc3}
  on
  the growth of $d_L$ and $d_R$ hold.  As before, let $\alpha=\frac{d_R}{d_L}$.
  
  We make precise the property of $G$ being locally a tree in the following
  lemma:
  \begin{lemma}\label{lem:tree}
    Let $r$ be fixed, and let $\tau$ be the set of vertices in $G$ whose
    $r$-neighborhoods contain no cycles.  Then, if $d_R$ satisfies
    \eqref{eq:dc1} and \eqref{eq:dc2},
    \begin{align*}
      \P\left[1-\frac{|\tau|}{n+m}>n^{-1/4}\right]=o\big(n^{-5/4}\big)
    \end{align*}
  \end{lemma}
  \begin{proof}
    This is the same statement proven in \cite{DuP} for regular graphs, 
    and using Proposition~\ref{prop:ev} we can prove it in the same way.
    If a vertex is not in $\tau$, then for some $s\leq r$ there exists
    a $2s$-cycle within $r-s$ of the vertex.  The size of all 
    $(r-s)$-neighborhoods of $2s$-cycles hence serves as a bound on the
    number of ``bad'' vertices.  For any given $2s$-cycle, the size
    of its $(r-s)$-neighborhood is bounded by $2s(d_R-1)^{r-s}$.
    If we define
    \begin{align*}
		  N_r^* = \sum_{s=2}^r2s(d_R-1)^{r-s}X_s,
    \end{align*}
    then this gives us the bound $n+m-|\tau|\leq N_r^*$.
    
    Now we compute $\E[N^*_r]$ and $\var[N^*_r]$.  Using our expression for
    $\E[X_r]$ from Proposition~\ref{prop:ev},
    \begin{align*}
		  \E[N_r^*] &= \sum_{s=2}^r 2s(d_R-1)^{r-s}(d_L-1)^s(d_R-1)^sO(1)\\
		     &=(d_R-1)^r\sum_{s=2}^rO\left((d_L-1)^s\right)\\
		     &=O((d_L-1)^r(d_R-1)^r)=O\big(d_R^{2r}\big) 
    \end{align*}
		To compute the variance, first notice that \eqref{eq:dc2} implies
    that $\var[X_s]= \mu_s(1+o(1))$.  By Cauchy-Schwarz,
		\begin{align*}
		  \var[N_r^*]&\leq r\sum_{s=2}^r 4s^2(d_R-1)^{2r-2s}\var[X_s]\\
		    &\leq r\sum_{s=2}^r 4s^2(d_R-1)^{2r-2s}\mu_s\left(1+
        o(1)\right)\\
        &=r(d_R-1)^{2r}\sum_{s=2}^rs\frac{(d_L-1)^s(d_R-1)^s}{(d_R-1)^{2s}}
          (1+o(1))\\
        &\leq r(d_R-1)^{2r}(1+o(1))\sum_{s=2}^rs\\
        &=O\big(d_R^{2r}\big)
    \end{align*}
    The rest of the lemma follows from Markov's inequality:
    \begin{align*}
      \P\left[1-\frac{|\tau|}{n+m}>n^{-1/4}\right]
        &=\P\big[n+m-|\tau|>(1+\alpha)n^{3/4}\big]\\
        &\leq \P\big[N_r^*>(1+\alpha)n^{3/4}\big]\\
        &\leq \frac{\var[N_r^*]+\E[N_r^*]^2}{(1+\alpha)^2n^{3/2}}\\
        &=O\left(\frac{d_R^{4r}}{n^{3/2}}\right)=o\big(n^{-5/4}\big).
        \tag*{\qedhere}
    \end{align*}
  \end{proof}
  This result shows that there are few ``bad'' vertices.
  It easily follows that this is true within the left and the right
  vertex classes of $G$ as well.
  \begin{cor}\label{cor:tree}
    Let $\tau_L$ and $\tau_R$ be the number of vertices in the left
    and right
    classes of $G$, respectively, with acyclic $r$-neighborhoods.
    Then
    \begin{align*}
      \P\left[\frac{m-|\tau_L|}{n+m}>n^{-1/4}\right]&=o(n^{-5/4}),\\
		  \P\left[\frac{n-|\tau_R|}{n+m}>n^{-1/4}\right]&=o(n^{-5/4}).
    \end{align*}
  \end{cor}
  \begin{proof}
    Note that $1-\frac{|\tau|}{n+m}=\frac{m-|\tau_L|}{n+m}
    +\frac{n-|\tau_R|}{n+m}$,
    so $1-\frac{|\tau|}{n+m}>c$ whenever $\frac{m-|\tau_L|}{n+m}>c$
    or $\frac{n-|\tau_R|}{n+m}>c$.
  \end{proof}
    
  Let $\beta_k(r,\sigma^2)$ be the $k$th moment of the
  Mar\v{c}enko-Pastur
  law with ratio $r$ and scaling factor $\sigma^2$ as defined
  in \cite{Bai}.
  \begin{prop}\label{thm:lln}
    Let $A$ be the adjacency matrix of $G$, and let
    $\mu_n$ be the ESD of $d_R^{-1/2}A$.
    Recalling that $y=\limn\alpha$, 
    \begin{align*}
      &\int x^{2k+1}\,d\mu_n(x)\to 0\ \text{a.s.,}\\
      &\int x^{2k}\,d\mu_n(x)\to \frac{2}{1+y}
        \beta_k(y^{-1},1)\ \text{a.s.}
    \end{align*}
    as $n\to\infty$.
  \end{prop}
  \begin{proof}
    Consider the infinite $(d_L,d_R)$-biregular tree.
    Let $B_r$ denote the number of closed walks of length $r$
    on this tree, starting from some fixed vertex of degree $d_L$,
    and let $C_r$ denote the number of closed walks of length $r$
    starting from a vertex of degree $d_R$.  Note that as $d_L$
    and $d_R$ depend on $n$, so do $B_r$
    and $C_r$.
    
    First, we formulate the $r$th moment of $\mu_n$ in terms of $B_r$ and
    $C_r$.
    \begin{align*}
      \int x^r\, d\mu_n(x) = \frac{d_R^{-r/2}}{n+m}\sum_{v\in V(G)}
         A^r(v,v).
    \end{align*}
      The quantity $A^r(v,v)$ is the number of closed walks
      of length $r$
      from $v$ in $G$.  With the same definitions
      of $\tau$, $\tau_L$, and $\tau_R$ as in Lemma~\ref{lem:tree}
      and Corollary~\ref{cor:tree},
      this is equal to
      $B_r$ when $v\in\tau_L$ and  $C_r$ when $v\in\tau_R$.
      For $v\not\in\tau$, we can use the bound $A^r(v,v)\leq
      d_R^r$.  Hence we can bound the $r$th moment
      of $\mu_n$ by
      \begin{align*}
        \frac{d_R^{-r/2}}{n+m}\big(|\tau_L|B_r+&|\tau_R|C_r\big)
        \leq \int x^r\, d\mu_n(x)\\&\leq
        \frac{d_R^{-r/2}}{n+m}\big(mB_r+nC_r+(n+m-|\tau|)d_R^r\big).
      \end{align*}
      Define
      \begin{align*}
        a_n &= d_R^{-r/2}\left(\Big(\frac{m}{n+m}-n^{-1/4}\Big)B_r+
        \Big(\frac{n}{n+m}-n^{-1/4}\Big)C_r\right),\\
        b_n &= d_R^{-r/2}\left(\frac{mB_r+nC_r}{n+m}
          +\frac{n^{-1/4}}{n+m}d_R^r\right).
      \end{align*}
      Now
      \begin{align*}
        \P\left[a_n\leq \int x^r\, d\mu_n(x)\right]&\geq
        \P\left[\text{$\frac{m-|\tau_L|}{n+m}\leq n^{-1/4}$
        and $\frac{n-|\tau_R|}{n+m}\leq n^{-1/4}$}\right]\\
        &=1-o(n^{-5/4}),
      \end{align*}
      and in the same way, $\P\left[\int x^r\, d\mu_n(x)\leq b_n\right]
      \geq 1-o(n^{-5/4})$.
      By the Borel-Cantelli~lemma, it holds almost surely that
      $a_n\leq\int x^r\, d\mu_n(x)\leq b_n$ for all but finitely many $n$.
      If we show that $a_n$ and $b_n$ converge to a common limit (which
      we will do next), 
      it will follow that 
      $\int x^r\, d\mu_n(x)$ converges to this limit almost surely.
          
      To find the limits of $a_n$ and $b_n$ as $n\to\infty$,
      we first note that $n^{-1/4}d^r_R \to 0$ by 
      \eqref{eq:dc2}.  Since $B_r,C_r\leq d_R^r$, this also
      implies that $n^{-1/4}B_r\to 0$ and $n^{-1/4}C_r\to 0$.
      So, it suffices to show that
      \begin{align}
        &\frac{d_R^{-(2k+1)/2}}{n+m}(mB_{2k+1}+nC_{2k+1})
           \to 0,\label{eq:suf1}\\
        &\frac{d_R^{-k}}{n+m}(mB_{2k}+nC_{2k})\to 
           \frac{2}{1+\alpha}
           \beta_k(y^{-1},1).\label{eq:suf2}
      \end{align}
      Equation \eqref{eq:suf1} is trivial, since $B_{2k+1}=C_{2k+1}=0$.
      To prove \eqref{eq:suf2}, we introduce the 
      \emph{Narayana numbers} (see \cite[p.~237]{EC2}),
      defined as
      \begin{align*}
        N(k,a)=\frac{1}{a+1}\binom{k}{a}\binom{k-1}{a}.
      \end{align*}
      The moments of the Mar\v{c}enko-Pastur
      distribution can be given in terms of these numbers
      \cite[Lemma~3.1]{Bai}:
      \begin{align}
        \beta_k(y^{-1},1) = \sum_{r=0}^{k-1}y^{-r}N(k,r).\label{eq:MPmoment}
      \end{align}
      We will give a combinatorial argument to relate the closed
      walks on the tree to the Narayana numbers.
      We mention that another approach to proving \eqref{eq:suf2}
      is to calculate $B_r$ and $C_r$ using the spectral density of the
      infinite $(d_L,d_R)$-biregular tree, as calculated in
      \cite[(5.7)]{God}.
    
    A \emph{Motzkin path} of length $2k$ is a lattice path that
    starts at $(0,0)$, ends at $(2k,0)$, and stays above the $x$-axis;
    each step can be a rise ($\nearrow$), a fall ($\searrow$),
    or a level step ($\rightarrow$).  An \emph{alternating Motzkin
    path} is a Motzkin path that rises only on even steps and that
    falls only on odd steps.  See Figure~\ref{fig:amps} for an example
    of the five alternating Motzkin paths of length 6.
    \begin{figure}
      \begin{center}
      \begin{tikzpicture}[scale=0.8]
        \begin{scope}[xshift=-3.75cm]
          \foreach \x in {0, 2, 4}
             \fill[xshift=\x cm,black!15!white] (-2,0) rectangle (-1,1);
          \draw[dashed,very thin] (-3,0) grid (3, 1);
          \draw[very thick] (-3,0) -- (3,0);
        \end{scope}  
        \begin{scope}[xshift=3.75cm]
          \foreach \x in {0, 2, 4}
             \fill[xshift=\x cm,black!15!white] (-2,0) rectangle (-1,1);
          \draw[dashed,very thin] (-3,0) grid (3, 1);
          \draw[very thick] (-3,0) -- (0,0) -- (1,1) -- (2,0) -- (3,0);
        \end{scope}  
        \begin{scope}[xshift=-3.75cm, yshift=-1.5cm]
          \foreach \x in {0, 2, 4}
             \fill[xshift=\x cm,black!15!white] (-2,0) rectangle (-1,1);
          \draw[dashed,very thin] (-3,0) grid (3, 1);
          \draw[very thick] (-3,0) -- (-2,0) -- (-1,1) -- (0,0) -- (3,0);
        \end{scope}  
        \begin{scope}[xshift=3.75cm, yshift=-1.5cm]
          \foreach \x in {0, 2, 4}
            \fill[xshift=\x cm,black!15!white] (-2,0) rectangle (-1,1);
          \draw[dashed,very thin] (-3,0) grid (3, 1);
          \draw[very thick] (-3,0) -- (-2,0) -- (-1,1) -- (0,0) --
            (1,1) -- (2, 0) -- (3,0);
        \end{scope}  
        \begin{scope}[xshift=-3.75cm,yshift=-3cm]
          \foreach \x in {0, 2, 4}
             \fill[xshift=\x cm,black!15!white] (-2,0) rectangle (-1,1);
          \draw[dashed,very thin] (-3,0) grid (3, 1);
          \draw[very thick] (-3,0) -- (-2,0) -- (-1,1) -- (1,1) -- (2,0)
             -- (3,0);
        \end{scope}  
      \end{tikzpicture}
      \end{center}\caption{The alternating Motzkin paths
      of length 6.}\label{fig:amps}
    \end{figure}
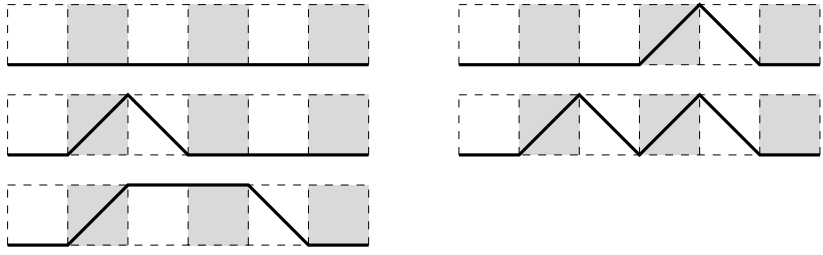
    
    The alternating Motzkin paths have the following connection
    to the Narayana numbers:
      \begin{lemma}[Lemma~6.1.7, \cite{Dum}]\label{lem:dum}
        The number of alternating Motzkin paths of length
        $2k$ with exactly $a$ rises is $N(k,a)$.
      \end{lemma}

      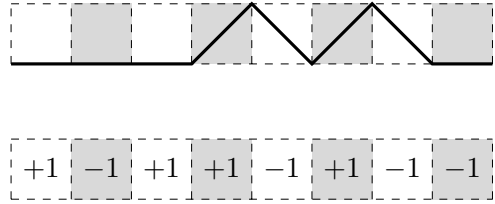
\begin{figure}
        \begin{center}
          \begin{tikzpicture}[scale=0.8]
            \foreach \x in {0, 2, 4, 6}
              \fill[xshift=\x cm,black!15!white] (-3,0) rectangle (-2,1);
            \draw[dashed,very thin] (-4,0) grid (4, 1);
            \draw[very thick] (-4,0) -- (-1,0) -- (0,1) -- (1,0)
              -- (2,1) -- (3,0) -- (4,0);
            \begin{scope}[yshift=-2.25cm]
              \foreach \x in {0, 2, 4, 6}
                \fill[xshift=\x cm,black!15!white] (-3,0) rectangle (-2,1);
              \draw[dashed,very thin] (-4,0) grid (4, 1);
              \draw (-3.5,0.5) node {$+1$};
              \draw (-2.5,0.5) node {$-1$};
              \draw (-1.5,0.5) node {$+1$};
              \draw (-0.5,0.5) node {$+1$};
              \draw (0.5,0.5) node {$-1$};
              \draw (1.5,0.5) node {$+1$};
              \draw (2.5,0.5) node {$-1$};
              \draw (3.5,0.5) node {$-1$};
            \end{scope}  
          \end{tikzpicture}
        \end{center}
        \caption{An alternating Motzkin path and its corresponding
        ballot sequence.}
      \end{figure}
    We relate the Narayana numbers to the walks on a tree
    by the following two lemmas.  A
    \emph{ballot sequence} of length $2k$
    is a sequence $x_1,\ldots,x_{2k}$ of $1$'s and $-1$'s 
    such that all partial sums $x_1+\cdots+x_j$
    are nonnegative.
    \begin{lemma}\label{lem:bal}
      The number of ballot sequences of length $2k$
      with $a$ 1's at even locations and $k-a$ 1's at odd locations
      is $N(k,a)$.
    \end{lemma}
    \begin{proof}
      We give a bijection between alternating
      Motzkin paths and ballot sequences.
      Encode the alternating Motzkin path as $p_1,\ldots,p_{2k}$,
      where $p_i$ is $1$, $0$, or $-1$ depending on whether
      the $i$th step is rising, level, or falling.
      Define a sequence by $x_i=2p_i+(-1)^{i-1}$.  We will confirm
      that this is a ballot sequence:
      each $x_i$ is either $1$ or $-1$;  for any $j$,
      \begin{align*}
        x_1+\cdots+x_i = 2(p_1+\cdots+p_i) + \sum_{i=1}^j(-1)^{i-1},
      \end{align*}
      and both of these terms are nonnegative;  
      and $x_1+\cdots+x_{2k}=2(p_1+\cdots+p_{2k})=0$.  So, $x_1,\ldots,
      x_{2k}$ is a ballot sequence.
      
      To map back from ballot sequences to alternating Motzkin paths,
      we let $p_i=(x_i-(-1)^{i-1})/2$.  For any even $j$,
      \begin{align*}
        p_1+\cdots+p_j = \frac{1}{2}(x_1+\cdots+x_j)\geq 0.
      \end{align*}
      For any odd $j$,
      \begin{align*}
        p_1+\cdots+p_j = \frac{1}{2}(x_1+\cdots+x_j-1).
      \end{align*}
      Since $x_1+\cdots+x_j\geq 1$ when $j$ is odd, this expression
      is also nonnegative.  So, our path
      stays above the $x$-axis.  The other properties of
      being an alternating Motzkin path are easy to check.
      
      This bijection takes
      alternating Motzkin paths with $a$ rises to ballot sequences
      with $a$ 1's at even locations, so the lemma follows
      from Lemma~\ref{lem:dum}.
    \end{proof}
    \newcommand{\ddL}{\widetilde{d}_L}
    \newcommand{\ddR}{\widetilde{d}_R}
    \newcommand{\ddLe}[1]{\widetilde{d}_L^{\;#1}}
    \newcommand{\ddRe}[1]{\widetilde{d}_R^{\;#1}}
    \begin{lemma}\label{lem:treenumbers}
      \begin{align}
        B_{2k}&=\sum_{a=0}^{k-1}(d_R-1)^a\ddLe{k-a}N(k,a),\label{eq:B}\\
        C_{2k}&=\sum_{a=0}^{k-1}(d_L-1)^a\ddRe{k-a}N(k,a),\label{eq:C}
      \end{align}
      for some $\ddL$ and $\ddR$ satisfying
       $d_L-1\leq\ddL\leq d_L$ and $d_R-1\leq\ddR\leq d_R$.
    \end{lemma}
    \begin{proof}
      Fix some vertex $v$ in the $(d_L,d_R)$-biregular tree with
      degree $d_L$ to serve as the root.
      We will enumerate the closed walks of length $2k$
      starting at $v$.  To any such walk we can associate a ballot
      sequence of length $2k$, given by putting a 1 at every
      step of the walk going away from $v$ and a $-1$ at every step returning
      toward $v$.  We will count the number of closed walks associated to
      each ballot sequence.
      
      Fix some ballot sequence, and suppose we are constructing a closed walk
      associated with it.
      For every 1 in the ballot sequence, our walk must go outward
      from the root.  If the 1 is at an even location,
      we have $d_R-1$ choices
      for where to move; if it is at an odd location, we have
      either $d_L$ or $d_{L}-1$, depending on whether we are moving
      from $v$ or from some other vertex.
      For every $-1$ in the ballot sequence, our walk must move
      backward towards the root, and there is no choice to be made.
      So, given a ballot sequence with $a$ rises on even steps,
      the number of closed walks from $v$ associated with that ballot sequence
      is between $(d_R-1)^a(d_L-1)^{k-a}$ and $(d_R-1)^ad_L^{k-a}$.
      Using Lemma~\ref{lem:bal} to count the number of ballot sequences
      with $r$ rises on even steps, we obtain \eqref{eq:B}.
      The same proof starting with a vertex $v$ with degree $d_R$
      gives us \eqref{eq:C}.
    \end{proof}
    
    Now we finish the proof of Proposition~\ref{thm:lln} by computing
    the limit of
    \begin{align*}
      \frac{d_R^{-k}}{n+m}\big(mB_{2k}+nC_{2k})
    \end{align*}
    as $n\to\infty$.
    Recall that all of these variables depend on $n$ except for $k$,
    which is fixed.
    By Lemma~\ref{lem:treenumbers}, we can rewrite the above expression as
    \begin{align*}
      \frac{m}{n+m}\sum_{r=0}^{k-1}\frac{(d_R-1)^{r}d_R^{-k}}{\ddLe{r-k}}N(k,r)
      +\frac{n}{n+m}\sum_{r=0}^{k-1}\frac{(d_L-1)^r}{\ddRe{r-k}d_R^{k}}
      N(k,r).
    \end{align*}
    Replacing $m/(n+m)$ and $n/(n+m)$ by $\alpha/(1+\alpha)$
    and $1/(1+\alpha)$ respectively, and taking the limit as $n\to\infty$
    yields
    \begin{align*}
      \frac{y}{1+y}\sum_{r=0}^{k-1}y^{r-k}N(k,r)
      +\frac{1}{1+y}\sum_{r=0}^{k-1}y^{-r}N(k,r).
    \end{align*}
    Note that this is where we used \eqref{eq:dc1}.
    We can simplify this expression to
    \begin{align*}
      \frac{2}{1+y}\sum_{r=0}^{k-1}y^{-r}N(k,r).
    \end{align*}
    This is exactly $\frac{2}{1+y}
      \beta_k(y^{-1},1)$ as given in \eqref{eq:MPmoment}.
  \end{proof}
  \begin{proof}[Proof of Theorem~\ref{thm:sq}]
    Let $\nu_n$ be the ESD of $d_R^{-1}X^TX$.
    As described in the introduction, the eigenvalues
    of $d_R^{-1/2}A$ consist of $\pm\sigma_i$ for the singular
    values $\sigma_1,\ldots,\sigma_n$ of $d_R^{-1/2}X$, along with $m-n$ 0's.
    It follows that
    \begin{align*}
      \int x^k\,d\nu_n &= \frac{m+n}{2n}\int x^{2k}\,d \mu_n\\
      &= \frac{\alpha+1}{2}\int x^{2k}\,d \mu_n.
    \end{align*}
    It follows from Proposition~\ref{thm:lln} and the convergence of
    $\alpha$ to $y$ that
    $\int x^k\,d\nu_n\to \beta_k(y^{-1},1)$ a.s.~as $n\to\infty$.
    Since the moments of $\nu_n$ converge almost surely
    to the moments of the Mar\v{c}enko-Pastur distribution,
    which is supported on a compact interval, $\nu_n$ converges
    almost surely to this distribution.
  \end{proof}
  
  To wrap things up, we compute the density of the limiting distribution
  of $\mu_n$.
  \begin{proof}[Proof of Corollary~\ref{cor:mudist}]
    We only need to show that the moments of the measure given
    by the density $(2/(1+y))p(x^2)|x|$ agree with the limits of the moments
    of $\mu_n$ found in Proposition~\ref{thm:lln}.
    The odd moments of this measure are 0 by symmetry,
    and the even moments are easily computed by integrating and
    substituting
    $u=x^2$.
  \end{proof}
  See Figure~\ref{fig:density} for a picture of the limiting distribution
  for a few values of $\alpha$. 
  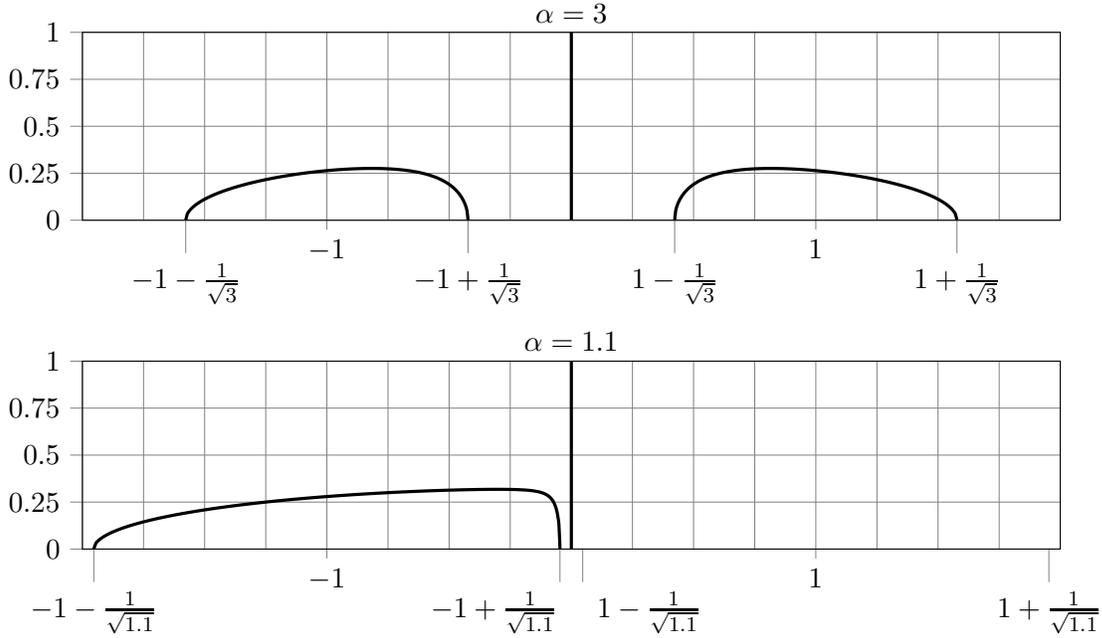
\begin{figure}
    \begin{center}
      \begin{tikzpicture}[yscale=2.5,xscale=3.25,help line/.style={very thin,gray}]
        \begin{scope}
        \foreach \x in {0,0.25,...,1.1}
          \draw[help line] (2,\x)--(-2.05,\x) node[left,black] {$\x$};
        \foreach \x in {-2,-1.75,...,2.1}
          \draw[help line] (\x,1)--(\x,0);
        \foreach \x in {-1,1}
          \draw[help line] (\x,0)--(\x,-0.05) node[below,black] {$\x$};
        \draw[help line] (-1.57735027,0)--(-1.57735027,-0.175) 
          node[below,black] {$-1-\frac{1}{\sqrt{3}}$};
        \draw[help line] (1.57735027,0)--(1.57735027,-0.175) 
          node[below,black] {$1+\frac{1}{\sqrt{3}}$};
        \draw[help line] (-0.422649731,0)--(-0.422649731,-0.175) 
          node[below,black] {$-1+\frac{1}{\sqrt{3}}$};
        \draw[help line] (0.422649731,0)--(0.422649731,-0.175) 
          node[below,black] {$1-\frac{1}{\sqrt{3}}$};
        \draw[very thick,smooth] plot file {dist-alpha-3-1.table};
        \draw[very thick,smooth] plot file {dist-alpha-3-2.table};
        \draw[very thick] (0,0)--(0,1.0);
        \draw (-2,0) rectangle (2,1);
        \draw (0,1)  node[above] {$\alpha=3$};
        \end{scope}
        \begin{scope}[yshift=-1.75cm]
        \foreach \x in {0,0.25,...,1.1}
          \draw[help line] (2,\x)--(-2.05,\x) node[left,black] {$\x$};
        \foreach \x in {-2,-1.75,...,2.1}
          \draw[help line] (\x,1)--(\x,0);
        \foreach \x in {-1,1}
          \draw[help line] (\x,0)--(\x,-0.05) node[below,black] {$\x$};
        \draw[help line] (-1.95346259,0)--(-1.95346259,-0.175) 
          node[below,black] {$-1-\frac{1}{\sqrt{1.1}}$};
        \draw[help line] (1.95346259,0)--(1.95346259,-0.175) 
          node[below,black] {$1+\frac{1}{\sqrt{1.1}}$};
        \draw[help line] (-0.0465374108,0)--(-0.0465374108,-0.175) 
          node[below,black,xshift=-25pt] 
            {$-1+\frac{1}{\sqrt{1.1}}$};
        \draw[help line] (0.0465374108,0)--(0.0465374108,-0.175) 
          node[below,black,xshift=25pt] {$1-\frac{1}{\sqrt{1.1}}$};
        \draw[very thick,smooth] plot file {dist-alpha-1.1-1.table};
        \draw[very thick,smooth] plot file {dist-alpha-1.1-2.table};
        \draw[very thick] (0,0)--(0,1.0);
        \draw (-2,0) rectangle (2,1);
        \draw (0,1)  node[above] {$\alpha=1.1$};
        \end{scope}
      \end{tikzpicture}
    \end{center}
    \caption{This figure depicts 
    the limiting density of the ESD of a random
    $(d_L,d_R)$-biregular bipartite graph for two different
    values of
    $\alpha=\frac{d_R}{d_L}$.  The spike at 0 denotes a point mass.
    The continuous part of the density
    is given by \eqref{eq:density}, and the point mass has size 
    $\frac{\alpha-1}{\alpha+1}$.
    Each of the left and right spikes are scaled down copies of the 
    Mar\^cenko-Pastur distribution under the transformation $x\mapsto \sqrt{x}$.
    When $\alpha=1$, the density
    reduces to that of the semicircle law.\label{fig:density}
    }
  \end{figure}
  \section{Convergence on short scales}\label{sec:locallaw}
  To show convergence of the ESD on short scales,
  we will use the method of \cite[Section~3]{DuP}.  The basic idea is
  to use the local approximation of our graph as a tree to estimate
  the graph's Stieltjes transform.
  
  First, we will define some terms and sketch the proof.
  The \emph{Stieltjes transform} of a probability measure $\mu$
  is the function $s(z) = \int (z-x)^{-1}\,d\mu(x)$ defined on the
  complex upper half-plane. 
  The Stieltjes transform of the ESD of an $n\times n$ Hermitian matrix $A$ is then
  $s(z)=\frac{1}{n}\tr R(z)$, where $R(z) = (A-zI)^{-1}$ is the \emph{resolvent}
  of $A$. If one can show that the Stieltjes tranform of the ESD converges, standard
  arguments from random matrix theory allow one to show that the ESD itself converges,
  with quantitative estimates on the Stieltjes transform translating into quantitative 
  estimates on the ESD.
  
  To simplify language, we will refer to the resolvent of a graph instead
  of the resolvent of the adjacency matrix of the graph. Similarly, we use
  the Stieltjes transform of a graph to mean the Stieltjes transform
  of the ESD of the adjacency matrix of the graph.
  Our goal is to show that the Stieltjes transform of
  a random biregular bipartite graph is close to its limit.
  We break the proof into the following steps:
  \begin{enumerate}
    \item (Section~\ref{subsec:trees}) Compute the resolvent matrix
      of a biregular
      tree of a given depth~$\zeta$.
    \item (Section~\ref{subsec:treestographs}) Let $v$ be a vertex of a deterministic biregular graph $G$
      with no cycles in its $(\zeta+1)$-neighborhood.  Show that the $(v,v)$ entry
      of the resolvent of $G$ is close to the
      $(\text{root}, \text{root})$ entry of the resolvent of a biregular tree of depth~$\zeta$.
    \item (Section~\ref{subsec:graphstorandgraphs})
      Show that nearly all the vertices of a random biregular graph
      have a large acyclic neighborhood, and use this fact to transfer
      the estimates of the earlier parts to random graphs, giving us
      an estimate of the Stieltjes transform.
  \end{enumerate}
  We finish by invoking a standard argument from 
 \cite[Lemma~64]{TaV} to deduce 
      Theorem~\ref{thm:locallaw} from the estimate on the Stieltjes transform.
      
  Our task is made slightly more difficult by the need to consider two
  different $(d_L,d_R)$-biregular trees: one in which the root has degree
  $d_L$, and one in which the root has degree $d_R$.
  \subsection{Preliminaries}

  We will use the following well-known formula for the inverse of
  a block matrix.
  \begin{prop}\label{eq:matrixinversion}
    Let $A$ and $D$ be $n\times n$ matrices of size $n\times n$
    and $m\times m$, respectively, and let $B$ be $n\times m$.
    Let
    \begin{align*}
      M = \begin{bmatrix} A&B\\B^T&D 
      \end{bmatrix}
    \end{align*}
    Then
    \begin{align*}
      M^{-1} &= \begin{bmatrix}
        A^{-1}+A^{-1}BF^{-1}B^TA^{-1} & -A^{-1}BF^{-1}\\
        -F^{-1}B^TA^{-1} & F^{-1}
      \end{bmatrix},&F = D-B^TA^{-1}B
    \end{align*}
  \end{prop}

  In this section of the paper, we will define the ratio $\alpha$
  by $\alpha = (d_R-1)/(d_L-1)$ rather than $d_R/d_L$.
  Let $U_n(z)$ be the Chebyshev polynomial of the second kind
  of degree $n$.
  We define the following shifted Chebyshev polynomial,
  \begin{align*}
    q_n(z) = \alpha^{-n/2}U_n\left(\frac{\sqrt{\alpha}(z^2-\alpha^{-1}-1)}{2}
      \right).
  \end{align*}
  This family of polynomials satisfies the recurrence
  \begin{align}
    q_{-1}(z) &= 0\nonumber\\
    q_0(z) &= 1\nonumber\\
    q_n(z) &= (z^2-\alpha^{-1}-1)q_{n-1}(z)-\alpha^{-1}q_{n-2}(z),\quad n\geq 1,
      \label{eq:qrec}
  \end{align}
  which follows by applying the recurrence
  $U_{n}(z) = 2zU_{n-1}(z)-U_{n-2}(z)$.
  
  \subsection{Resolvents of trees}\label{subsec:trees}
  Our aim is to calculate the resolvents of biregular trees.
  We start, however, by considering trees in which each vertex has
  either $d_L-1$ or $d_R-1$ children; this means that every vertex
  has degree $d_L$ or $d_R$ except for the root, which has degree $d_L-1$
  or $d_R-1$.
  
  We define $T_L(\zeta)$ to be the tree with depth $\zeta$ where the root
  has $d_L-1$ children, its children each have $d_R-1$ children, their
  children each have $d_L-1$ children, and so on.  The tree $T_L(0)$ is
  a single vertex.  We define $T_R(\zeta)$ similarly, but with the root
  having $d_R-1$ children.  To determine the adjacency matrices of
  these trees, we place an ordering on the vertices as follows: 
  If $\zeta=0$, then there is only one vertex and hence one possible
  labeling.  For $\zeta>0$, we will define an ordering inductively.  Choose
  a subtree of the root and list of all its vertices in the order
  already determined for $\zeta-1$.  
  Then, do this with the remaining subtrees of the root.
  Finally, put the root last.
  
  Let $H_L$ and $H_R$ be the adjacency matrices of $T_L(\zeta)$
  and $T_R(\zeta)$, respectively.  We define
  \begin{align*}
    \varphi_L(\zeta) &= ((d_R-1)^{-1/2}H_L-z)^{-1}_{\text{root,root}}\\
    \psi_L(\zeta) &= ((d_R-1)^{-1/2}H_L-z)^{-1}_{\text{root,leaf}}
  \end{align*}
  We note that $\psi_L(\zeta)$ is independent of the particular leaf chosen.
  We will make use of the recursive structure of the trees to
  calculate these values.
  \begin{lemma}\label{lem:almostreg}
    \begin{enumerate}[(a)]
    \item
    \begin{align*}
      \varphi_L(2\zeta) &= -\frac{q_\zeta(z) + \alpha^{-1}q_{\zeta-1}(z)}
        {zq_\zeta(z)},
      & \varphi_L(2\zeta+1) &= -\frac{zq_\zeta(z)}{q_{\zeta+1}(z)+
        q_\zeta(z)},\\
      \varphi_R(2\zeta) &= -\frac{q_\zeta(z) + q_{\zeta-1}(z)}
        {zq_\zeta(z)},
      & \varphi_R(2\zeta+1) &= -\frac{zq_\zeta(z)}{q_{\zeta+1}(z)+
        \alpha^{-1}q_\zeta(z)},\\
    \end{align*}
    \item
      \begin{align*}
        \psi_L(2\zeta) &=\psi_R(2\zeta) = 
            -\frac{(d_R-1)^{-\zeta}}{zq_\zeta(z)}\\
        \psi_L(2\zeta+1) &= -\frac{(d_R-1)^{-\zeta-1/2}}
           {q_{\zeta+1}(z)+q_\zeta(z)},&
         \psi_R(2\zeta+1) &= -\frac{(d_R-1)^{-\zeta-1/2}}
           {q_{\zeta+1}(z)+\alpha^{-1}q_\zeta(z)}.
      \end{align*}
    \end{enumerate}
  \end{lemma}
  \begin{proof}
    We will start by showing that $\varphi_L(\zeta)$
    and $\varphi_R(\zeta)$ satisfy the
    recurrences
    \begin{align}
      \varphi_L(\zeta) &= 
        -\big(z+\alpha^{-1}\varphi_R(\zeta-1)\big)^{-1},\label{eq:rec1}\\
      \varphi_R(\zeta) &=
        -\big(z+\varphi_L(\zeta-1)\big)^{-1}\label{eq:rec2}.
    \end{align}
    Consider the tree $T_L$ of depth $\zeta$
    and let $H_1,\ldots,H_{d_L-1}$ denote the adjacency
    matrices of the subtrees of the root.
    Using the given ordering for the vertices, we have
    \begin{align*}
      \frac{1}{\sqrt{d_R-1}}H_L(\zeta) - z = 
        \begin{bmatrix}
          \frac{1}{\sqrt{d_R-1}}H_R(\zeta-1)-z &&&&\\
          & \frac{1}{\sqrt{d_R-1}}H_R(\zeta-1)-z&&&\\
          &&\ddots&&u \\
          &&&\frac{1}{\sqrt{d_R-1}}H_R(\zeta-1)-z\\
          &u^T&&&-z
        \end{bmatrix}
    \end{align*}
    where $u$ is a column vector representing the children of the root.
    This vector is $(d_R-1)^{-1/2}$ in the root of each of the subtrees of
    the root and 0 elsewhere.  Using Proposition~\ref{eq:matrixinversion}
    and thinking of the $-z$ in the bottom right corner as a $1\times 1$
    block, we find
    \begin{align*}
      \varphi_L(\zeta) = \left(-z-\frac{d_L-1}{d_R-1}
      \varphi_R(\zeta-1)\right)^{-1}
    \end{align*}
    which is \eqref{eq:rec1}.  The proof for \eqref{eq:rec2} is the same.
    
    Unwinding these recurrences and noting
    that $\varphi_L(0) = \varphi_R(0)=-z^{-1}$, 
    we have the following continued fraction
    representation of $\varphi_L(\zeta)$:
    \begin{align*}
      \varphi_L(\zeta) &= -\cfrac{1}{z - \cfrac{\alpha^{-1}}
        {z - \cfrac{1}{z-\cfrac{\cdots}{{z-z^{-1}}}}}}.
    \end{align*}
    Using standard formulas for the evaluation of continued fractions (see
    \cite{LoWaa}),
    we find that $\varphi_L(\zeta)=\frac{A_\zeta}{B_\zeta}$, where
    \begin{align*}
      A_{2\zeta} &= zA_{2\zeta-1}-A_{2\zeta-2}, &
      A_{2\zeta+1} &= zA_{2\zeta}-\alpha^{-1}A_{2\zeta-1}\\
      B_{2\zeta} &= zB_{2\zeta-1}-B_{2\zeta-2}, &
      B_{2\zeta+1} &= zB_{2\zeta}-\alpha^{-1}B_{2\zeta-1}
    \end{align*}
    with the initial conditions
    \begin{align*}
      A_0 &= -1, & A_1 &= -z,\\
      B_0 &= z, & B_1 &= z^2-\alpha^{-1}.
    \end{align*}
    We can iterate these recurrences as follows:
    \begin{align*}
      A_{2\zeta} &= z(zA_{2\zeta-2}-\alpha^{-1}A_{2\zeta-3}) - A_{2\zeta-2}\\
          &= z\big(zA_{2\zeta-2}-\frac{\alpha^{-1}}{z}(A_{2\zeta-2}+
          A_{2\zeta-4})\big)
           - A_{2\zeta-2}\\
         &= (z^2-\alpha^{-1}-1)A_{2\zeta-2}-\alpha^{-1}A_{2\zeta-4}.
    \end{align*}
    Applying this procedure to the $A_{2\zeta+1}$ and to the
    $B_{2\zeta}$ and $B_{2\zeta+1}$ cases give the same result, yielding
    \begin{align*}
      A_{\zeta} &= (z^2-\alpha^{-1}-1)A_{\zeta-1}-\alpha^{-1}A_{\zeta-2},\\
      B_{\zeta} &= (z^2-\alpha^{-1}-1)B_{\zeta-1}-\alpha^{-1}B_{\zeta-2}.
    \end{align*}
    It is easily checked using \eqref{eq:qrec} that
    \begin{align*}
      A_{2\zeta} &= -\big(q_{\zeta}(z)+\alpha^{-1}q_{\zeta-1}(z)\big), &
      A_{2\zeta+1} &= -zq_{\zeta}(z),\\
      B_{2\zeta} &= zq_{\zeta}(z), &
      B_{2\zeta+1} &= q_{\zeta+1}(z)+q_\zeta(z).
    \end{align*}
    From these expressions and \eqref{eq:rec2}, it is straightforward to derive
    the expressions for $\varphi_R(2\zeta)$ and $\varphi_R(2\zeta+1)$.      
    
    To compute $\psi_L(2\zeta)$ and $\psi_R(2\zeta)$, we will first show
    \begin{align}
      \psi_L(\zeta) &= -(d_R-1)^{-1/2}\varphi_L(\zeta)\psi_R(\zeta-1),
      \label{eq:psilrec}\\
      \psi_R(\zeta) &= -(d_R-1)^{-1/2}\varphi_R(\zeta)\psi_L(\zeta-1)
      \label{eq:psirrec}.
    \end{align}
    Using Proposition~\ref{eq:matrixinversion}, if $A$ is the minor
    of $(d_R-1)^{-1/2}H_L-z$ consisting of all but the last row
    and column,
    we have
    \begin{align*}
      \psi_L(\zeta)&= \left(\frac{1}{\sqrt{d_R-1}}H_L-z\right)^{-1}_{1,
      \text{root}}\\
        &=-\varphi_L(\zeta)(A^{-1}u)_1\\
        &=-(d_R-1)^{-1/2}\varphi_L(\zeta)\psi_R(\zeta-1),
    \end{align*}
    proving \eqref{eq:psilrec}.
    Equation \eqref{eq:psirrec} is derived in the same way.  By combining
    these,
    we get $\psi_L(\zeta)=(d_R-1)^{-1}\varphi_L(\zeta)\varphi_R(\zeta-1)
    \psi_L(\zeta-2)$, whence
    \begin{align*}
      \psi_L(2\zeta) &= \psi_L(0)(d_R-1)^{-\zeta}\prod_{j=1}^{\zeta}
        \varphi_L(2j)\varphi_R(2j-1)\\
       &=-z^{-1}(d_R-1)^{-\zeta}\prod_{j=1}^\zeta \frac{q_{j-1}(z)}{q_j(z)}\\
       &=-\frac{(d_R-1)^{-\zeta}}{zq_\zeta(z)}.
    \end{align*}
    The proof of the expression for $\psi_R(2\zeta)$ is identical, and
    the expressions for $\psi_L(2\zeta+1)$ and $\psi_R(2\zeta+1)$
    follow immediately by applying \eqref{eq:psilrec} and
    \eqref{eq:psirrec}.
  \end{proof}

  We now turn from these almost regular trees to the real thing.
  Let $\widetilde{T}_L(\zeta)$ be the $(d_L,d_R)$-biregular
  tree of depth $\zeta$ whose root has degree $d_L$,
  and let $\widetilde{T}_R(\zeta)$ be the $(d_L,d_R)$-biregular
  tree of depth $\zeta$ whose root has degree $d_R$.  Let $\widetilde{H}_L$ and
  $\widetilde{H}_R$ be the adjacency matrices of $\widetilde{T}_L(\zeta)$
  and $\widetilde{T}_R(\zeta)$, respectively, with vertices ordered
  as with $H_L$ and $H_R$.
  We define
  \begin{align*}
    \widetilde{\varphi}_L(\zeta) &= 
      ((d_R-1)^{-1/2}\widetilde{H}_L-z)^{-1}_{\text{root,root}},&
      \widetilde{\varphi}_R(\zeta) &= 
      ((d_R-1)^{-1/2}\widetilde{H}_R-z)^{-1}_{\text{root,root}},\\
    \widetilde{\psi}_L(\zeta) &= ((d_R-1)^{-1/2}
    \widetilde{H}_L-z)^{-1}_{\text{root,leaf}},&
    \widetilde{\psi}_R(\zeta) &= ((d_R-1)^{-1/2}
    \widetilde{H}_R-z)^{-1}_{\text{root,leaf}}.
  \end{align*}
  \begin{lemma}\label{lem:regtree}
    \begin{align*}
      \widetilde{\varphi}_L(2\zeta) 
      &= -\frac{q_\zeta(z) + \alpha^{-1}q_{\zeta-1}(z)}
        {z(q_\zeta(z)-\frac1{d_R-1}q_{\zeta-1}(z))},&
      \widetilde{\varphi}_R(2\zeta) 
      &= -\frac{q_\zeta(z) + q_{\zeta-1}(z)}
        {z(q_\zeta(z)-\frac1{d_R-1}q_{\zeta-1}(z))}\\
      \widetilde{\psi}_L(2\zeta) &=\widetilde{\psi}_R(2\zeta) =
        -\frac{(d_R-1)^{-\zeta}}{z(q_\zeta(z)-
      \frac{1}{d_R-1}q_{\zeta-1}(z))}.
    \end{align*}
  \end{lemma}
  \begin{proof}
    Because the root of $\widetilde{H_L}$ has $d_L$ children
    instead of $d_L-1$, the methods of Lemma~\ref{lem:almostreg}
    give
    \begin{align*}
      \widetilde{\varphi}_L(\zeta) &= 
        \left(-z-\frac{d_L}{d_R-1}
        \varphi_R(\zeta-1)\right)^{-1}.
    \end{align*}
    Substituting in the value for $\varphi_R(\zeta-1)$
    from Lemma~\ref{lem:almostreg} yields the desired expression.
    The expression for $\widetilde{\varphi}_R(2\zeta)$ is derived in the
    same way.
    
    The same procedure shows that
    \begin{align*}
      \widetilde{\psi}_L(2\zeta)=-(d_R-1)^{-1/2}\widetilde{\varphi}_L(2\zeta)
      \psi_R(2\zeta-1),
    \end{align*}
    and substituting the value from Lemma~\ref{lem:almostreg} yields 
    the desired expression.
  \end{proof}

  We will now bound the rate of convergence of some of
  these functions to their limits
  as $\zeta\to\infty$.  First, 
  define the complex function $F(z) = z+\sqrt{z^2-1}$, with
  branch cut $[0,\infty)$ for the square root. Let
  $w(z) = F\bigl(\frac{1}{2}\sqrt{\alpha}(z^2-\alpha^{-1}-1)\bigr)$ and $r(z)=|w(z)|$.
  We will refer to $w(z)$ and $r(z)$ as simply $w$ and $r$.
  Note that $r>1$ for all $\Im(z)>0$.  Using a well-known expression for 
  the Chebyshev function $U_n(z)$ (see \cite{MH}), we can
  expand $q_{\zeta}(z)$ as
  \begin{align}
    q_{\zeta}(z) &= 
       \alpha^{-\zeta/2}\frac{w^{\zeta+1}-w^{-\zeta-1}}{w-w^{-1}}.\label{eq:chebw}
  \end{align}
  
  As we will see, the limits of $\varphi_L(\zeta)$
  and $\varphi_R(\zeta)$ as $\zeta\to\infty$
  are given by the following functions,
  defined on the upper half-plane:
  \begin{align*}
    s_L(z) &= -\frac1z - \frac{\alpha^{-1/2}w^{-1}}{z},&
    s_R(z) &=-\frac1z - \frac{\alpha^{1/2}w^{-1}}{z},
  \end{align*}
  with branch cut $[0,\infty)$ for the square root.  We define
  \begin{align*}
    s(z) &= \frac{\alpha s_L(z)+s_R(z)}{1+\alpha}
      = \frac{\alpha}{1+\alpha}\left(
      -z + \sqrt{\left(z-\frac{1-\alpha}{\alpha z}\right)^2-4},
      \right)
  \end{align*}
  again with branch cut $[0,\infty)$ for the square root,
  and we note that $s(z)$ is the Stieltjes transform of the limiting ESD
  $\mu$ of Corollary~\ref{cor:mudist}.
  
  We bound the convergence in terms of $r=r(z)$ and $\zeta$:
  \begin{lemma}\label{lem:bounds}
  \begin{align*}
    |\varphi_L(2\zeta) - s_L(z)|
      &\leq 
      \frac{2\alpha^{-1/2}r^{-2\zeta}}
      {|z|(1-r^{-2\zeta-2})}\\
    |\varphi_R(2\zeta) - s_R(z)|
      &\leq 
      \frac{2\alpha^{1/2}r^{-2\zeta}}
      {|z|(1-r^{-2\zeta-2})}\\
    |\psi_L(2\zeta)|,\ |\psi_R(2\zeta)| &\leq
    \frac{2\alpha^{\zeta/2}(d_R-1)^{-\zeta}r^{-\zeta}
      }{|z|(1-r^{-2\zeta-1})}
  \end{align*}
  \end{lemma}
  
  \begin{proof}
    Applying \eqref{eq:chebw} to the formula for $\varphi_L(2\zeta)$,
    \begin{align*}
      |\varphi_L(2\zeta) - s_L(z)| &= \frac{\alpha^{-1/2}}{|z|}
        \left|w^{-1} - \frac{w^{\zeta}-w^{-\zeta}}{w^{\zeta+1}-w^{-\zeta-1}}\right|\\
        &=\frac{\alpha^{-1/2}}{|z|}\left|\frac{w^{-2\zeta-1}(1-w^{-2})}{1 - 
          w^{-2\zeta-2}}\right|\\
        &\leq \frac{2\alpha^{-1/2}r^{-2\zeta-1}}{|z|(1-r^{-2\zeta-2})}.
    \end{align*}
    The exact same procedure establishes the corresponding inequality
    for $|\varphi_R(2\zeta)-s_R(2\zeta)|$.  We can similarly compute
    \begin{align*}
      |\psi_L(2\zeta)|
      = |\psi_R(2\zeta)|
      &= \left| \frac{(d_R-1)^{-\zeta}(w-w^{-1})}{z(w^{\zeta+1}-w^{-\zeta-1})}
      \right|\\
      &\leq \frac{2\alpha^{\zeta/2}(d_R-1)^{-\zeta}r^{-\zeta}}
         {|z|(1-r^{-2\zeta-2})}.\tag*{\qedhere}
    \end{align*}
  \end{proof}
  \subsection{From trees to deterministic graphs}\label{subsec:treestographs}
  We now move from trees to graphs with large, acyclic neighborhoods.
  Let $G$ be a deterministic $(d_L,d_R)$-biregular 
  graph that has a root vertex with an acyclic $(\zeta+1)$-neighborhood.
  Let
  $A$ be the adjacency matrix of this graph.
  Our goal is to show that the resolvent of $A$ is well approximated by
  the resolvent of $\widetilde{H}_L$ or $\widetilde{H}_R$.
  If the root of $G$ has degree $d_L$, then we consider the error term
  \begin{align*}
    E_L(\zeta) &= \big((d_R-1)^{-1/2}A-z\big)^{-1}_{\text{root,root}}
      - \widetilde{\varphi}_L(\zeta),
  \end{align*}
  and if the root of $G$ has degree $d_R$, then we consider the error term
  \begin{align*}
    E_R(\zeta) &= \big((d_R-1)^{-1/2}A-z\big)^{-1}_{\text{root,root}}
      - \widetilde{\varphi}_R(\zeta).
  \end{align*}
  
  \begin{lemma}\label{lem:Ebound}
    We can bound $E_L$ and $E_R$ by
    \begin{align*}
      |E_L(\zeta)|
      &\leq 
       \frac{|\widetilde{\psi}_L(\zeta)|^2d_L(d_R-1)^{\ceil{\zeta/2}}
        (d_L-1)^{\floor{\zeta   /2}}}
      {(d_R-1)\Im(z)}\\
      |E_R(\zeta)|
      &\leq \frac{|\widetilde{\psi}_R(\zeta)|^2
      d_R(d_L-1)^{\ceil{\zeta/2}}
        (d_R-1)^{\floor{\zeta/2}}}
      {(d_R-1)\Im(z)}
    \end{align*}
  \end{lemma}
  \begin{proof}
    Let $C_L$ and $C_R$ denote the number of vertices
    of distance $\zeta+1$ from the root in the biregular
    trees $\widetilde{T}_L$ and $\widetilde{T}_R$, respectively.
    We claim that
    \begin{align}
      |E_L(\zeta)| &\leq \frac{\big|\widetilde{\psi}_L(\zeta)\big|^2
        C_L}{(d_R-1)\Im(z)},      \label{eq:ELER1}\\
      |E_R(\zeta)| &\leq \frac{\big|\widetilde{\psi}_R(\zeta)\big|^2
        C_R}{(d_R-1)\Im(z)}.
      \label{eq:ELER2}
    \end{align}
    These statements can be proven exactly as in
    \cite[Lemma~10]{DuP}; we will only give a sketch here.
    The gist of the argument is to partition the vertices
    of $G$ into two parts, those at distance $\zeta$ or less from the root and those at distance
    greater than $\zeta$. This decomposes $A$ into blocks, one of which is simply $\widetilde{H}_L$
    or $\widetilde{H}_R$. An application of Proposition~\ref{eq:matrixinversion} and some calculations
    then prove \eqref{eq:ELER1} and \eqref{eq:ELER2}.
    The rest is simply a calculation of $C_L$ and $C_R$.
  \end{proof}
  Now, we take a sequence of graphs as above and let $\zeta$ grow to infinity.
  \begin{lemma}\label{lem:detgraphs}
    Let $G$ be a sequence of deterministic graphs, each with a root
    with an acyclic $2\zeta+1$ neighborhood.
    Suppose $d_R\to\infty$ and $\zeta\to\infty$.  Fix $\eps>0$, and let $z$
    be a sequence with $|\Re(z)|>\eps$ and $|\Im(z)|
    \geq 1/d_R$.  Suppose that $r^{-2\zeta}
    =o(1/d_R^2)$.   Then either
    \begin{align*}
      \left|\big((d_R-1)^{-1/2}A - z\big)^{-1}_{\text{root,root}} - s_L(z)
      \right| &= O_\eps(1/d_R),\\
      \intertext{or}
      \left|\big((d_R-1)^{-1/2}A - z\big)^{-1}_{\text{root,root}} - s_R(z)
      \right| &= O_\eps(1/d_R),
    \end{align*}
    depending on whether the root of $G$ has degree $d_L$ or $d_R$.
    We use $O_\eps(\cdot)$ to indicate that the constant in the expression
    depends on $\eps$.
  \end{lemma}
  \begin{proof}
    Consider the case where the root of $G$ has degree $d_L$.
    We will proceed in two steps, bounding first $|\widetilde{\varphi}_L(2\zeta)
    -s_L(z)|$ and then $E_L(2\zeta)$.
    
    We define the quantity
    \begin{align*}
      \beta &= \frac{q_{2\zeta}(z)}{q_{2\zeta}(z)-(d_R-1)^{-1}q_{2\zeta-1}(z)}\\
        &=\frac{1}{1-(d_R-1)^{-1}\frac{w^{-1}(1-w^{-2\zeta})}{1-w^{-2\zeta-2}}}
    \end{align*}
    This is of interest because $\widetilde{\varphi}_L(2\zeta)=
    \beta\varphi_L(2\zeta)$.
    With the assumptions of this lemma, one can calculate directly
    that $|\beta| = 1+O(1/d_R)$. By these assumptions
    and Lemma~\ref{lem:bounds},
    $|\varphi_L(2\zeta)-s_L(z)|=o_{\eps}(1/d_R)$.
    Since $s_L(z)$ is bounded for $|z|>\eps$, this also
    implies that $\varphi_L(2\zeta)=O_{\eps}(1)$.
    Thus
    \begin{align*}
      |\widetilde{\varphi}_L(2\zeta)-s_L(z)|&=
       \big|\beta\varphi_L(2\zeta)-s_L(z)\big|\\
       &\leq |\varphi_L(2\zeta)-s_L(z)|+ |(\beta-1)\varphi_L(2\zeta)|\\
       &\leq o_\eps(1/d_R) + O_\eps(1/d_R) = O_\eps(1/d_R).
    \end{align*}
    
    Since $\widetilde{\psi}_L(2\zeta)=\beta\psi_L(2\zeta)$, 
    by Lemma~\ref{lem:bounds} and our bound on $\beta$,
    \begin{align*}
      |\widetilde{\psi}_L(2\zeta)| = \alpha^{\zeta/2}(d_R-1)^{-\zeta}o_\eps(1/d_R).
    \end{align*}
    Combining this with our bound on $E_L(2\zeta)$ 
    from Lemma~\ref{lem:Ebound} gives
    \begin{align*}
      |E_L(2\zeta)| &\leq\frac{o_\eps(1/d_R^2)}{\Im(z)}=o_\eps(1/d_R).
    \end{align*}
    These two bounds prove the lemma.  The case when the root of $G$
    has degree $d_R$ is the same.
  \end{proof}
  \subsection{From deterministic graphs to random graphs}\label{subsec:graphstorandgraphs}
  The main actors of this section will be sequences $s$, $\zeta$,
  and $\eta$.  We will choose $s$ and $z$ 
  in such a way that $s\geq r(z)$, and $\zeta$ will represent
  the size of an acyclic neighborhood in the graph.  We will
  choose $\eta$ so that we can control the Stieltjes transform
  of our graph
  on the set $U=\{z:\ \Im(z)\geq\eta\}$.  The following lemma
  gives us the relation between $s$ and $\eta$:
  \begin{lemma}\label{lem:ellipse}
    Let $r=r(z)$, fix some $s>1$, and let
    $\eta = s^{1/2}-s^{-1/2}$.  If $\Im(z)\geq\eta$, then
    $r\geq s$.  
  \end{lemma}
  \begin{proof}
    First we prove this when $\alpha=1$.
    Consider the set
    $E_{s}=\{z:\ |F(z)|<s\}$.
    This set is the interior of an ellipse whose foci
    are $-1$ and $1$ and whose radii are $\frac{1}{2}(s+s^{-1})$
    and $\frac{1}{2}(s-s^{-1})$ (see \cite[p.~14]{MH}).
    It suffices to show that if $\Im(z)\geq \eta$, then 
    $\frac{1}{2}z^2-1$ lies outside of $E_{s}$.
    To this end, we note that the transformation $z\mapsto \frac{1}{2}
    z^2-1$ takes the region given by $\Im(z)\geq \eta$ to the region
    bounded on the right by the parabola $P=\{\frac12(t^2-\eta^2)-1+
    \eta ti: t\in\RR\}$.  This can be checked to touch $E_{s}$ at
    $-\frac12(s+s^{-1})$ and otherwise to lie to the left of it.
    This proves the lemma when $\alpha=1$.
    
    To extend this to the case where $\alpha>1$, we consider the image
    of $E_{s}$ under the map $z\mapsto \big(\alpha^{-1/2}(2z+\alpha^{1/2}
    +\alpha^{-1/2})\big)^{1/2}$, which is the inverse of $\frac12\sqrt{\alpha}
    (z^2-\alpha^{-1}-1)$.  Our argument for $\alpha=1$
    establishes that in this case, the maximum imaginary part
    of this set is $\eta$.  It is straightforward to check that
    for any $z$, the quantity $\Im\big(\alpha^{-1/2}(2z+\alpha^{1/2}
    +\alpha^{-1/2})\big)^{1/2}$ decreases as $\alpha$ increases, which
    establishes the lemma.
  \end{proof}
  
  For the remainder of this section, let $G$ be a random biregular
  bipartite graph on $m+n$ vertices satisfying \eqref{eq:dc1}--\eqref{eq:dc3}
  as well as the condition $d_R = \exp\big(o(1)\sqrt{\log n}\big)$.
  Let $A$ be the adjacency matrix of $G$.
  We define the sequences
  \begin{align*}
    a &= \min\left(\frac{\log n}{9(\log d_R)^2}, d_R\right),\\
    s &= e^{1/a},\\
    \zeta &= \frac{\log n}{8\log d_R}-1,\\
    \eta &= s^{1/2}-s^{-1/2}.
  \end{align*}
  
  We now show that sufficiently many vertices of $G$ have tree-like
  neighborhoods.
  \begin{lemma}
    Let $J$ be the set of vertices in $G$ whose $2\zeta$-neighborhoods
    are acyclic.  Then  \label{lem:tree2}
    \begin{align*}
      \P\left[1-\frac{|J|}{n+m}\geq \frac{\eta}{d_R}\right] = o(1/n).
    \end{align*}
  \end{lemma}
  \begin{proof}
    This is nearly the same as Lemma~\ref{lem:tree}.  
    We may assume $\zeta$ is an integer by replacing it with $\floor{\zeta}$.
    We define
    \begin{align*}
      N^* &= \sum_{i=2}^{2\zeta}2i(d_R-1)^{2\zeta-i}X_i,
    \end{align*}
    recalling that $X_i$ is the random variable denoting the number of
    $2i$-cycles in $G$.  We have the bound
    $n+m-|J|\leq N^*$.  Now we apply Proposition~\ref{prop:ev} to calculate
    \begin{align*}
      \E[N^*] &= \sum_{i=2}^{2\zeta}2i(d_R-1)^{2\zeta-i}\mu_i\left(1+O\Big(
      \frac{i(i+d_R)}{n}\Big)\right)\\
      &=O\big(d_R^{4\zeta}\big),\\
      \intertext{and}
      \var[N^*] &\leq 2\zeta\sum_{i=2}^{2\zeta}4i^2(d_R-1)^{4\zeta-2i}\mu_i
         \left(1+O\left(\frac{d_R^{2i}(i\alpha^{2i-1}+\alpha^{-i}d_R)}
           {n}\right)\right)\\
         &\leq2\zeta\sum_{i=2}^\zeta 2i(d_R-1)^{4\zeta}
           \left(1+O\left(\frac{n^{4c}(\zeta\alpha^{4\zeta-1}
           +d_R)}{n}\right)\right)\\
         &= O\big(\zeta^3d_R^{4\zeta}\big).
    \end{align*}
    By Markov's inequality,
    \begin{align*}
      \P\left[1-\frac{|J|}{n+m}\geq \frac{\eta}{d_R}\right]
        &\leq \P\left[N^* \geq \frac{(1+\alpha)n\eta}{d_R}\right]\\
        &\leq \frac{O\big(d_R^{8\zeta+2}+\zeta^3d_R^{4\zeta+2}\big)}
          {n^2\eta^2}\\
        &\leq O\big(n^{-1}d_R^{-4} + n^{-3/2}\zeta^3\big)=
          o(1/n).\tag*{\qedhere}
    \end{align*}
  \end{proof}
  Let $J_L$ and $J_R$ denote the sets of vertices with acyclic
  $2\zeta$-neighborhoods in the left and right vertex classes, respectively.
  As in Corollary~\ref{cor:tree}, it is immediate that
  \begin{align*}
    \P\left[\frac{m-|J_L|}{n+m}\geq\frac{\eta}{d_R}\right]&=o(1/n)\\
    \intertext{and}
    \P\left[\frac{n-|J_R|}{n+m}\geq\frac{\eta}{d_R}\right]&=o(1/n).
  \end{align*}
  It is also straightforward to see that this lemma holds
  when we require the vertices to have acyclic $2\zeta+1$ neighborhoods
  rather than just $2\zeta$ neighborhoods.

  We can now apply all of these results to the task of bounding
  the rate of convergence of the Stieltjes transform:
  \begin{thm}\label{thm:stieltjescontrol}
    Fix some $\eps$ and let
    $U$ denote the set of complex numbers 
    \begin{align*}
      U = \{z\in\CC:\ \Re(z)\geq\eps,\ \Im(z)\geq\eta\}.
    \end{align*}
    Let $s_n(z)$ denote the Stieltjes transform
    of $(d_R-1)^{-1}A$.  Then for sufficiently large $C_\eps$,
    \begin{align*}
      \P\left[\sup_{z\in U}|s_n(z)-s(z)|>C_\eps/d_R\right]=o(1/n).
    \end{align*}
  \end{thm}
  \begin{proof}
    Let $J$ denote the vertices with acyclic $2\zeta+1$ neighborhoods.
    We condition on the event that $(m-|J_L|)/(m+n)<\eta/d_R$ and
    $(n-|J_R|/(m+n)<\eta/d_R$, which
    by the discussion following
    Lemma~\ref{lem:tree2} holds with probability $1-o(1/n)$.
    We call this event $\Omega$.
    We compute $s_n(z)$ for $z\in U$, breaking up the vertices into
    $J_L$, $J_R$, and the remaining vertices:
    \begin{align*}
      s_n(z) &= \frac{1}{m+n}
       \sum_{v\in J_L}\big((d_R-1)^{-1/2}A-z\big)^{-1}_{v,v}
               + \frac{1}{m+n} \sum_{v\in J_R}\big((d_R-1)^{-1/2}
                  A-z\big)^{-1}_{v,v}\\
               &\phantom{=}+ \frac{1}{n+m} \sum_{v\not\in J}
                  \big((d_R-1)^{-1/2}A-z\big)^{-1}_{v,v}.
    \end{align*}
    We begin with the third term.  Applying the bound
    $\big((d_R-1)^{-1/2}A-z\big)^{-1}_{v,v}\leq \eta^{-1}$, we have
    \begin{align}
      \left|\frac{1}{n+m} \sum_{v\not\in J}
                  \big((d_R-1)^{-1/2}A-z\big)^{-1}_{v,v}\right|
            &\leq \frac{m+n-|J|}{(m+n)\eta}<\frac{1}{d_R}\label{eq:badverts}
    \end{align}
    on the event $\Omega$.
    
    Since every vertex in $J$ has an acyclic $2\zeta+1$ neighborhood, we will
    apply Lemma~\ref{lem:detgraphs} to estimate the first two terms.
    First, we confirm that the conditions of the lemma hold.
    By expanding $\eta$ as a power series, we deduce the bound
    $\eta\geq 1/a\geq 1/d_R$.  We can calculate
    \begin{align*}
      s^{-2\zeta} &= \exp\left(-\frac2a\left(\frac{\log n}{8\log d_R}-1\right)
          \right)\\
          &\leq \exp\left(-\frac{9}{4}\log d_R + o(1)\right)=o(1/d_R^2).
    \end{align*}
    By Lemma~\ref{lem:ellipse}, it follows from $\Im(z)\geq\eta$ that
    $r(z)\geq s$.  Hence for any sequence $z\in U$, we have
    $r(z)^{-2\zeta}\leq s^{-2\zeta}=o(1/d_R^2)$.
    Thus the conditions of Lemma~\ref{lem:detgraphs} hold,
    and so for all $v\in J_L$,
    \begin{align*}
      \left|\big((d_R-1)^{-1/2}A-z\big)^{-1}_{v,v}-s_L(z)\right|=O_\eps(1/d_R),
    \end{align*}
    and for all $v\in J_R$,
    \begin{align*}
      \left|\big((d_R-1)^{-1/2}A-z\big)^{-1}_{v,v}-s_R(z)\right|=O_\eps(1/d_R).
    \end{align*}
    Combining these estimates
    with \eqref{eq:badverts},
    \begin{align*}
      s_n(z)&=\frac{|J_L|}{m+n}s_L(z)+\frac{|J_R|}{m+n}s_R(z) + O_\eps(1/d_R)\\
        &= \frac{m}{m+n}s_L(z)+\frac{n}{m+n}s_R(z) + O_\eps(1/d_R)\\
        &= s(z)+O_\eps(1/d_R)
    \end{align*}
    on the event $\Omega$, which proves the theorem.
  \end{proof}

  We now restate and prove our local convergence law.
  \begin{restate}
    Fix $\eps>0$.  Let $\mu_n$ be the ESD
    of $(d_R-1)^{-1/2}A$, and let $\mu$ be the limiting ESD defined
    in
    Corollary~\ref{cor:mudist}.
    There exists a constant $C_\eps$
    such that for all sufficiently large $n$ and $\delta>0$,
    for any interval $I\subset\RR$ 
    avoiding
    $[-\eps,\eps]$ and with length $|I|\geq\max\big(2\eta,\eta/(-\delta
    \log\delta)\big)$, it holds that
    \begin{align*}
      |\mu_n(I)-\mu(I)|<\delta C_\eps|I|
    \end{align*}
    with probability $1-o(1/n)$.
  \end{restate}
  \begin{proof}
    This theorem follows from the arguments of \cite[Lemma~64]{TaV},
    which we will sketch.  Define
    \begin{align*}
      F(y) = \frac{1}{\pi}\int_I\frac{\eta}{\eta^2+(y-x)^2}\,dx.
    \end{align*}
    This function $F(y)$ approximates the indicator function on the interval
    $I$, and the following statements hold:
    \begin{align*}
      \int F(y)\,d\mu(y) &= \mu(I) + O_\eps\left(\eta\log\frac{|I|}{\eta}\right),\\
      \int F(y)\,d\mu_n(y) &= \mu_n(I)+ O_\eps\left(\eta\log\frac{|I|}{\eta}\right).
    \end{align*}
    The proofs of these statements in \cite[Lemma~64]{TaV} have $\mu$
    as the semicircle law, but they apply just as well to our limiting
    measure $\mu$; the only thing necessary for the proof to go through
    is that $\mu$ has a bounded density outside of the interval
    $[-\eps,\eps]$.  On the event $\Omega$ of the previous theorem,
    \begin{align*}
      \left|\int F(y)\,d\mu(y)-\int F(y)\,d\mu_n(y)\right|
        &=\frac{1}{\pi}\left|\int_I\big(\Im(s(x+\eta i))-\Im(s_n(x+\eta i))\big)
          \,dx\right|\\
        &\leq \frac{C_\eps|I|}{\pi d_R}.
    \end{align*}
    As observed in \cite{TaV}, it follows from the condition
    $|I|\geq \eta/(-\delta\log\delta)$ that
    $\eta\log\frac{|I|}{\eta}=O(\delta|I|)$.
    Since $d_R\to\infty$,  
    for $n$ sufficiently large,
    \begin{align*}
      \left|\mu(I)-\mu_n(I)\right| &\leq
        C_\eps\delta|I|
    \end{align*}
    for some constant $C_\eps$ (not necessarily the same one as before)
    on the event $\Omega$.
  \end{proof}

  \section*{Appendix}

  Our goal here is to prove Proposition~\ref{prop:ev}.
  We mention that it is possible to prove much more than this.
  The main result of \cite{MWW} is that the distribution
  of short cycles in a random regular graph is approximately Poisson,
  and this result holds for biregular bipartite graphs as well, with suitable
  modifications of the proofs.
  
  We will use a theorem from \cite{McK} that gives us the probability
  that $G$ contains some subgraph $L\subset K_{m,n}$.
  For any $v\in K_{m,n}$, let $g_v$ and $l_v$ denote the
  the degree of $v$ considered as a vertex of $G$ and of $L$, respectively.
  Let $l_{\max}$ be the largest value of $l_i$.
  Consider $L$ to be a collection of edges, so that $|L|$ is the
  number of edges of $L$.
  The notation $[x]_a$ denotes the falling factorial,
  $x(x-1)\cdots(x-a+1)$.
  \begin{prop}\label{thm:mcgbound}
    Let $L\subset K_{m,n}$.
    \begin{enumerate}[(a)]
    \item\label{item:gub}
    If $|L|+2d_R(d_R+l_{\max}-2)\leq nd_R-1$, then
    \begin{align*}
      \P[L\subset G]&\leq\frac{\prod [g_i]_{l_i}}{[nd_R-4d_R^2-1]_{|L|}}.
    \end{align*}    
    \item\label{item:glb}
      If $nd_R-2d_R(d_R+l_{\max}-1)-1-|L|\geq d_Rl_{\max}$, then
      \begin{multline*}
        \P[L\subset G]\geq\frac{\prod[g_i]_{l_i}}{[nd_R-1]_{|L|}}\\
        {}\times\left[\left(1-\frac{d_Rl_{\max}}{nd_R-|L|-2d_R(d_R+l_{\max}-1)-1}
        \right)/\left(1+\frac{d_R^2}{nd_R-2d_R(d_R+l_{\max}-2)-1-|L|(e-1)/e}
        \right)
        \right]^{|L|}
      \end{multline*}
    \end{enumerate}
  \end{prop}
  \begin{proof}
  This is an application of \cite[Theorem~3.5]{McK}, with the set 
  $H$ from that theorem equal to $\emptyset$.
  \end{proof}
  
  We are most interested in when $L$ is a cycle with $2r$ edges, in which
  case the above theorem reduces to the following:
    \begin{cor}\label{thm:mcbound}
    Let $L\subset K_{m,n}$ be a 
    cycle with $2r$ edges whose presence in $G$
    we wish to test.  
    \begin{enumerate}[(a)]
      \item\label{item:ub}
        If $2r+2d_R^2\leq nd_R-1$, then
        \begin{align*}
          \P[L\subset G]&\leq
          \frac{ d_L^r(d_L-1)^rd_R^r(d_R-1)^r}{[nd_R-2d_R^2-1]_{2r}}.
        \end{align*}      
      \item\label{item:lb}
        If $nd_R-2d_R(d_R+1)-1-2r\geq 2d_R$, then
        \begin{multline*}
          \P[L\subset G]\geq\frac{ d_L^r(d_L-1)^rd_R^r(d_R-1)^r}
            {[nd_R-1]_{2r}}\\
           {}\times\left[\left(1-\frac{2d_R}{nd_R-2d_R(d_R+1)-1-2r}\right)\Big/
             \left(1+\frac{d_R^2}{nd_R-2d_R^2-1-2r(e-1)/e}\right)
             \right]^{2r}.
        \end{multline*}
    \end{enumerate}
  \end{cor}
    
  \begin{proof}[Proof of Proposition~\ref{prop:ev}]
    Let $L\subset K_{m,n}$ be a cycle with $2r$ edges.
    We start by showing that
    \begin{align}
      \P[L\subset G]= \frac{(d_L-1)^r(d_R-1)^r\alpha^{-r}}{n^{2r}}
      \left(1 + O\left(\frac{rd_R}{n}+\frac{r^2}{nd_R}\right)
      \right).\label{eq:pbound}
    \end{align}

    Since $d_R=o(n)$, all of the conditions for
    Proposition~\ref{thm:mcgbound} and Corollary~\ref{thm:mcbound} apply.  By
    Corollary~\ref{thm:mcbound}\ref{item:ub},
    \begin{align}
        \P[L\subset G]&\leq\frac{(d_L-1)^r(d_R-1)^r\alpha^{-r}d_R^{2r}}
        {(nd_R-2d_R^2-2r)^{2r}}\nonumber\\
        &=\frac{(d_L-1)^r(d_R-1)^r\alpha^{-r}}{n^{2r}}\left(
        \frac{nd_R}{nd_R-2d_R^2-2r}\right)^{2r}\nonumber\\
        &=\frac{(d_L-1)^r(d_R-1)^r\alpha^{-r}}{n^{2r}}\left(1+
        \frac{2d_R^2+2r}{nd_R-2d_R^2-2r}\right)^{2r}\nonumber\\
        &=\frac{(d_L-1)^r(d_R-1)^r\alpha^{-r}}{n^{2r}}\left(
        1+O\left(\frac{rd_R}{n}+\frac{r^2}{nd_R}\right)\right).\label{eq:ubasym}
    \end{align}
    The last line follows from the fact that if $x>-1$,
    \begin{align*}
      (1+x)^r\leq e^{rx}=1+O(rx)
    \end{align*}
    as $rx\to 0$.
    
    From Corollary~\ref{thm:mcbound}\ref{item:lb},
      \begin{align*}
      \begin{split}
        \P[L\subset G]&\geq\frac{d_L^r(d_L-1)^rd_R^r(d_R-1)^r}
          {(nd_R)^{2r}}\\
        &\quad
        \times\left[\left(1-\frac{2d_R}{nd_R-2d_R(d_R+1)-1-2r}\right)/
        \left(1+\frac{d_R^2}{nd_R-2d_R^2-1-2r(e-1)/e}\right)
        \right]^{2r}
      \end{split}\\
      &= \frac{(d_L-1)^r(d_R-1)^r\alpha^{-r}}
          {n^{2r}}\left(\frac{1+O(1/n)}{1+O(d_R/n)}\right)^{2r}\\
      &=\frac{(d_L-1)^r(d_R-1)^r\alpha^{-r}}
          {n^{2r}}\biggl(1+O\biggl(\frac{d_R}{n}\biggr)\biggr)^{2r}\\
      &=\frac{(d_L-1)^r(d_R-1)^r\alpha^{-r}}
          {n^{2r}}\biggl(1+O\biggl(\frac{rd_R}{n}\biggr)\biggr).
      \end{align*}
      These two inequalities prove \eqref{eq:pbound}.
      
      The number of cycles of length $2r$ in $K_{m,n}$ is 
      $[m]_r[n]_r/2r$.  Using $[n]_r=n^r(1+O(r^2/n))$, which can
      be proven by showing inductively that $[n]_r\geq
      n^r(1-r^2/2n)$,
      we calculate the expected number of such cycles:
      \begin{align}
        \E[X_r] &=  \frac{[m]_r[n]_r}{2r}
        \frac{(d_L-1)^r(d_R-1)^r\alpha^{-r}}{n^{2r}}
        \left(1 + O\left(\frac{rd_R}{n}+\frac{r^2}{nd_R}\right)
        \right)\nonumber\\
        &=\frac{(n\alpha)^rn^r(1+O(r^2/n))}{2r}
        \frac{(d_L-1)^r(d_R-1)^r\alpha^{-r}}{n^{2r}}
        \left(1 + O\left(\frac{rd_R}{n}+\frac{r^2}{nd_R}\right)
        \right)\nonumber\\
        &=
        \mu_r
        \left(1 + O\left(\frac{r(r+d_R)}{n}\right)\right).\label{eq:easym}
      \end{align}

      Let $\Cc$ be the set of $2r$-cycles in $K_{m,n}$.
      We will calculate $\var[X_r]$ using the equation
      \begin{align*}
        \E[X_r^2] = \sum_{C_1\in\Cc}\sum_{C_2\in\Cc}\P[C_1\cup C_2\subset G].
      \end{align*}
      
      We break up $\Cc\times\Cc$ to help calculate this sum.
      \begin{align*}
        \Cc_1 &= \{(C_1,C_2)\in \Cc\times\Cc:\ C_1\cap C_2=\emptyset\},\\
        \Cc_2 &= \{(C_1,C_2)\in\Cc\times\Cc:\ \text{$C_1\cap C_2\neq\emptyset$,
          but $C_1\neq C_2$}\},\\
        \Cc_3 &= \{(C_1,C_2)\in\Cc\times\Cc:\ C_1=C_2\}.
      \end{align*}
      We are considering cycles as collections of edges, so pairs
      of cycles that share vertices but not edges belong in $\Cc_1$
      rather than $\Cc_2$.
      
      For $(C_1,C_2)\in\Cc_1$, Proposition~\ref{thm:mcgbound}\ref{item:gub}
      and a calculation
      identical to the one in \eqref{eq:ubasym} show that
      \begin{align*}
        \P[C_1\cup C_2\subset G]\leq
          \frac{(d_L-1)^{2r}(d_R-1)^{2r}\alpha^{-2r}}{n^{4r}}\left(
        1+O\left(\frac{rd_R}{n}+\frac{r^2}{nd_R}\right)\right)
      \end{align*}
      Bounding $|\Cc_1|$ by $|\Cc|=([m]_r[n]_r/2r)^2$ and
      repeating the calculation in \eqref{eq:easym} gives
      \begin{align}
        \sum_{(C_1,C_2)\in\Cc_1}\P[C_1\cup C_2\subset G] \leq \mu_r^2
        \Big(1+O\Big(\frac{r(r+d_R)}{n}\Big)\Big).\label{eq:C1sum}
      \end{align}
      For a lower bound on this sum, we note that by 
      Proposition~\ref{thm:mcgbound}\ref{item:glb}, for any $2r$-cycles
      $C_1$ and $C_2$ that share no vertices,
      \begin{align*}
        \P[C_1\cup C_2\subset G]&\geq 
        \frac{d_L^{2r}(d_L-1)^{2r}d_R^{2r}(d_R-1)^{2r}}{(nd_R)^{4r}}
        \left(
        \frac{1+O(1/n)}
        {1+O(d_R/n)}
        \right)^{4r}\\
        &= \frac{(d_L-1)^{2r}(d_R-1)^{2r}\alpha^{-2r}}{n^{4r}}
        \left(1+O\left(\frac{rd_R}{n}\right)
        \right).
      \end{align*}
      Summing this over the $([n]_r[m]_r/2r)([n-r]_r[m-r]_r/2r)$ such
      pairs of $2r$-cycles provides the bound
      \begin{align*}
        \sum_{(C_1,C_2)\in\Cc_1}\P[C_1\cup C_2\subset G] &\geq        
          \frac{[n]_r[m]_r[n-r]_r[m-r]_r\alpha^{-2r}}{n^{4r}}
          \mu_r^2
        \left(1+O\left(\frac{rd_R}{n}\right)
        \right)\\
        &=
          \frac{n^rm^r(n-r)^r(m-r)^r(1+O(r^2/n))
          \alpha^{-2r}}{n^{4r}}
          \mu_r^2
        \left(1+O\left(\frac{rd_R}{n}\right)
        \right)\\
        &=
          \biggl(\left(1-\frac{r}{n}\right)\left(1-\frac{r}{\alpha n}\right)
          \biggr)^r
          \mu_r^2
          \left(1+O\left(\frac{r^2}{n}\right)\right)
        \left(1+O\left(\frac{rd_R}{n}\right)
        \right)\\
        &=
          \mu_r^2
        \left(1+O\left(\frac{r(r+d_R)}{n}\right)
        \right).
      \end{align*}
      
      The sum over $\Cc_3$ is
      \begin{align}
        \sum_{(C_1,C_2)\in\Cc_3}\P[C_1\cup C_2\subset G] = \E[X_r]=
        \mu_r\Big(1+O\Big(\frac{r(r+d_R)}{n}\Big)\Big).\label{eq:C3sum}
      \end{align}
      
      To estimate the sum over $\Cc_2$, we bound the number of
      isomorphism types of a graph $H=C_1\cup C_2$ for
      $(C_1,C_2)\in \Cc_2$.  Let $H'$ be the graph
      $(V(C_1)\cap V(C_2),E(C_1)\cap E(C_2))$.
      Say that $H'$ has $p$ components and $j$ edges.  As $H'$ is a forest,
      it has $p+j$ vertices, so $H$ has $4r-p-j$ vertices.
      We also note that $H$ has
      $4r-j$ edges.
      
      Let $C_1=a_1a_2\cdots a_{2r}$ and $C_2=b_1b_2\cdots b_{2r}$,
      with $a_1=b_1$.  Let $A_1,\ldots,A_p$ be the components of $H'$
      ordered as the appear in $C_1$.  We can encode the isomorphism
      type of  $H$ in the following four sequences:
      \begin{itemize}
        \item $s_i$ is the number of vertices in $A_i$
        \item $t_i$ is the smallest $j$ such that $a_j\in A_i$
        \item $u_i$ is the smallest $j$ such that $b_j\in A_i$
        \item
          $v_i$ specifies whether $A_i$ is oriented the same way
          in $C_1$ as in $C_2$ (if $A_i$ is a single vertex, 
          consider it to be oriented the same)
      \end{itemize}
      For example, the following diagram is the union of two cycles of
      length 8, the first colored black and the second
      colored gray.
      
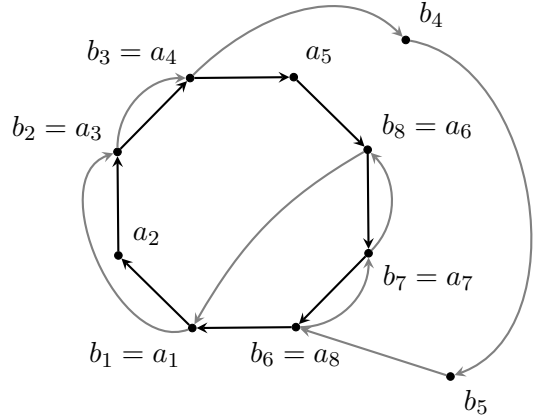
\begin{center}
          \begin{tikzpicture}[scale=1.8,vert/.style={circle,fill,inner sep=0,
              minimum size=0.1cm,draw},>=stealth]
            \node[vert,label=248:{$b_1=a_1$}] (a1) at (248:1) {};
            \node[vert,label=10:{$a_2$}] (a2) at (203:1) {};
            \node[vert,label=158:{$b_2=a_3$}] (a3) at (158:1) {};
            \node[vert,label=113:{$b_3=a_4$}] (a4) at (113:1) {};
            \node[vert,label=68:{$a_5$}] (a5) at (68:1) {};
            \node[vert,label=23:{$b_8=a_6$}] (a6) at (23:1) {};
            \node[vert,label=-22:{$b_7=a_7$}] (a7) at (-22:1) {};
            \node[vert,label=270:{$b_6=a_8$}] (a8) at (-67:1) {};
            \node[vert,label=45:{$b_4$}] (b4) at (45:1.7) {};
            \node[vert,label=-40:{$b_5$}] (b5) at (-40:2) {};
            \begin{scope}[thick]
              \draw[->] (a1) to (a2);
              \draw[->] (a2) to (a3);
              \draw[->] (a3) to (a4);
              \draw[->] (a4) to (a5);
              \draw[->] (a5) to (a6);
              \draw[->] (a6) to (a7);
              \draw[->] (a7) to (a8);
              \draw[->] (a8) to (a1);
            \end{scope}
            \begin{scope}[thick,gray]
              \draw[->] (a1) to [out=200,in=200] (a3);
              \draw[->] (a3) to [out=90,in=180] (a4);
              \draw[->] (a4) to [out=45,in=135] (b4);
              \draw[->] (b4) to [out=350,in=20] (b5);
              \draw[->] (b5) to (a8);
              \draw[->] (a8) to [out=0,in=270] (a7);
              \draw[->] (a7) to [out=45,in=315] (a6);
              \draw[->] (a6) to [out=210,in=60] (a1);
            \end{scope}
          \end{tikzpicture}
        \end{center}
      The intersection graph $H'$ has three components, 
      $A_1=a_1=b_1$, $A_2=a_3a_4=b_2b_3$, and $A_3=a_6a_7a_8=b_6b_7b_8$.  The
      four sequences for this graph are
      \begin{align*}
        \mathbf{s}:&\ 1, 2, 3\\
        \mathbf{t}:&\ 1, 3, 6\\
        \mathbf{u}:&\ 1, 2, 6\\
        \mathbf{v}:&\ \text{yes, yes, no}
      \end{align*}
      To illustrate that the isomorphism class of $H$ is encoded
      in these sequences, We will demonstrate how to recover it in
      this example.
      Start by drawing a cycle and labeling its vertices $a_1,\ldots,
      a_8$.  From $\mathbf{s}$ and $\mathbf{t}$, we can deduce that
      $A_1=a_1$, $A_2=a_3a_4$, and $A_3=a_6a_7a_8$.  From
      $\mathbf{u}$ and $\mathbf{v}$, we deduce that $b_1=a_1$,
      $b_2=a_3$, $b_3=a_4$, $b_6=a_8$, $b_7=a_7$, and $b_8=a_6$.
      Since $b_4$ and $b_5$ are unaccounted for, we conclude that
      they are not contained in $a_1\cdots a_8$.  Once we add edges
      to connect $b_1$ to $b_2$, $b_2$ to $b_3$, and so on, 
      we have recreated $H$ up to
      isomorphism.
        
      Now, we consider the number of possible isomorphism classes of 
      some $H=C_1\cup C_2$.
      The sequence $\mathbf{s}$ is a composition of $p+j$ into
      exactly $p$ parts, so there are $\binom{p+j-1}{p-1}$
      possibilities.  We know that $t_1=1$, and we know that
      $t_2,\ldots,t_p$ are ordered and are distinct, so
      there are at most $\binom{2r-1}{p-1}$ choices for $\mathbf{t}$.
      We know that $u_1=1$ and that $u_2,\ldots,u_p$ are distinct but not
      necessarily in any order, so there are at most
      $\binom{2r-1}{p-1}(p-1)!$ choices for $\mathbf{u}$.  For
      $\mathbf{v}$, there are $2^{p-1}$ choices.  For a fixed choice of
      of $p$ and $j$, the number of possible isomorphism classes of
      $H$ is hence bounded by
      \begin{align*}
        \binom{p+j-1}{p-1}\binom{2r-1}{p-1}^2(p-1)!2^{p-1}.
      \end{align*}
      By replacing $p+j-1$ and $2r-1$ by $2r$, we bound this quantity
      by
      \begin{align*}
        \binom{2r}{p-1}^3(p-1)!2^{p-1}\leq \frac{(16r^3)^{p-1}}{(p-1)!^2}.
      \end{align*}
      
      Suppose an isomorphism type has $a$ vertices from the left vertex class
      and $b$ from the right vertex class (with $a+b=4r-p-j$).
      Then this isomorphism type can be realized in at most
      $[n]_a[m]_b+[n]_b[m]_a\leq 2\alpha^{4r-p-j} n^{4r-p-j}$ ways.
      By Proposition~\ref{thm:mcgbound}\ref{item:gub}, 
      the probability of any one of these
      realizations being a subgraph of $G$ is bounded by
      \begin{align*}
        \frac{d_L^{(4r-j)/2}(d_L-1)^{(4r-j)/2}
        d_R^{(4r-j)/2}(d_R-1)^{(4r-j)/2}}{[nd_R-4d_R^2-1]_{4r-j}}=
        \frac{(d_L-1)^{(4r-j)/2}(d_R-1)^{(4r-j)/2}}{\alpha^{(4r-j)/2}n^{4r-j}}O(1).
      \end{align*}
      All together, we have the bound
      \begin{align}
        \sum_{(C_1,C_2)\in\Cc_3}\P[C_1\cup C_2\subset G]
        &\leq \sum_{p=1}^{2r}\sum_{j=1}^{2r}\frac{(16r^3)^{p-1}}{(p-1)!^2}
        \alpha^{4r-p-j} n^{4r-p-j}\frac{(d_L-1)^{(4r-j)/2}(d_R-1)^{(4r-j)/2}}
        {\alpha^{(4r-j)/2}n^{4r-j}}O(1)\nonumber\\
        &= \sum_{p=1}^{2r}\sum_{j=1}^{2r}\frac{(16r^3)^{p-1}}{(p-1)!^2}
        \frac{(d_L-1)^{(4r-j)/2}(d_R-1)^{(4r-j)/2}}
        {\alpha^{p-2r-j/2}n^{p}}O(1)\nonumber\\
        &=O\left(\frac{\alpha^{(4r-1)/2}
        (d_L-1)^{(4r-1)/2}(d_R-1)^{(4r-1)/2}}{n}\right).\label{eq:sharededge}
      \end{align}
      
      Combining this with \eqref{eq:C1sum} and \eqref{eq:C3sum} shows
      \begin{align*}
        \begin{split}
        \var[X_r]&= \mu_r^2
        \Big(1+O\Big(\frac{r(r+d_R)}{n}\Big)\Big)+\mu_r
        \Big(1+O\Big(\frac{r(r+d_R)}{n}\Big)\Big)\\
        &\phantom{=}\ -\E[X_r]^2 +
        O\left(\frac{\alpha^{(4r-1)/2}
        (d_L-1)^{(4r-1)/2}(d_R-1)^{(4r-1)/2}}{n}\right)
        \end{split}\\
        &=\mu_r+\mu_r^2\left(
        O\Big(\frac{r(r+d_R)}{n}\Big)+
        O\Big(\frac{r(r+d_R)}{\mu_rn}\Big)+
        O\left(\frac{r^2\alpha^{(4r-1)/2}}{n(d_L-1)^{1/2}(d_R-1)^{1/2}}
        \right)\right)\\
        &=\mu_r+\mu_r^2O\left(\frac{r\bigl(\alpha^{(4r-1)/2}r
          +d_R\bigr)}{n}\right).
      \end{align*}
      Hence
      \begin{align*}
        \var[X_r]
          &=\mu_r\left(1+\mu_rO\left(\frac{r\bigl(\alpha^{(4r-1)/2}r
          +d_R\bigr)}{n}\right)\right)\\
          &=\mu_r\left(1+O\bigg(\frac{d_R^{2r}(r\alpha^{2r-1}+\alpha^{-r}d_R)}
          {n}\bigg)\right).
          \tag*{\qedhere}
      \end{align*}
  \end{proof}

  \bibliographystyle{alpha}
  \bibliography{paper-bip}

\newcommand{\etalchar}[1]{$^{#1}$}
\begin{thebibliography}{MWW04}

\bibitem[BL13]{BL}
Shimon Brooks and Elon Lindenstrauss.
\newblock Non-localization of eigenfunctions on large regular graphs.
\newblock {\em Israel Journal of Mathematics}, 193(1):1--14, 2013.

\bibitem[BS10]{Bai}
Zhidong Bai and Jack~W. Silverstein.
\newblock {\em Spectral analysis of large dimensional random matrices}.
\newblock Springer Series in Statistics. Springer, New York, second edition,
  2010.

\bibitem[DJPP13]{DJPP}
Ioana Dumitriu, Tobias Johnson, Soumik Pal, and Elliot Paquette.
\newblock Functional limit theorems for random regular graphs.
\newblock {\em Probab. Theory Related Fields}, 156(3--4):921--975, 2013.

\bibitem[DP12]{DuP}
Ioana Dumitriu and Soumik Pal.
\newblock Sparse regular random graphs: {S}pectral density and eigenvectors.
\newblock {\em Ann. Probab.}, 40(5):2197--2235, 2012.

\bibitem[Dum03]{Dum}
Ioana Dumitriu.
\newblock {\em Eigenvalue Statistics for Beta-Ensembles}.
\newblock PhD thesis, Massachussetts Institute of Technology, 2003.

\bibitem[EPR{\etalchar{+}}10]{ESY3}
L{\'a}szl{\'o} Erd{\H{o}}s, Sandrine P{\'e}ch{\'e}, Jos{\'e}~A. Ram{\'{\i}}rez,
  Benjamin Schlein, and Horng-Tzer Yau.
\newblock Bulk universality for {W}igner matrices.
\newblock {\em Comm. Pure Appl. Math.}, 63(7):895--925, 2010.

\bibitem[ERS{\etalchar{+}}10]{ESY5}
L{\'a}szl{\'o} Erd{\H{o}}s, Jos{\'e} Ram{\'{\i}}rez, Benjamin Schlein, Terence
  Tao, Van Vu, and Horng-Tzer Yau.
\newblock Bulk universality for {W}igner {H}ermitian matrices with
  subexponential decay.
\newblock {\em Math. Res. Lett.}, 17(4):667--674, 2010.

\bibitem[ERSY10]{ESY4}
L{\'a}szl{\'o} Erd{\H{o}}s, Jos{\'e}~A. Ram{\'{\i}}rez, Benjamin Schlein, and
  Horng-Tzer Yau.
\newblock Universality of sine-kernel for {W}igner matrices with a small
  {G}aussian perturbation.
\newblock {\em Electron. J. Probab.}, 15:no. 18, 526--603, 2010.

\bibitem[ESY09a]{ESY2}
L{\'a}szl{\'o} Erd{\H{o}}s, Benjamin Schlein, and Horng-Tzer Yau.
\newblock Local semicircle law and complete delocalization for {W}igner random
  matrices.
\newblock {\em Comm. Math. Phys.}, 287(2):641--655, 2009.

\bibitem[ESY09b]{ESY1}
L{\'a}szl{\'o} Erd{\H{o}}s, Benjamin Schlein, and Horng-Tzer Yau.
\newblock Semicircle law on short scales and delocalization of eigenvectors for
  {W}igner random matrices.
\newblock {\em Ann. Probab.}, 37(3):815--852, 2009.

\bibitem[GM88]{God}
C.~D. Godsil and B.~Mohar.
\newblock Walk generating functions and spectral measures of infinite graphs.
\newblock In {\em Proceedings of the {V}ictoria {C}onference on {C}ombinatorial
  {M}atrix {A}nalysis ({V}ictoria, {BC}, 1987)}, volume 107, pages 191--206,
  1988.

\bibitem[Joh14]{J}
Tobias Johnson.
\newblock Exchangeable pairs, switchings, and random regular graphs.
\newblock Preprint. Available at arXiv:1112.0704, 2014.

\bibitem[JP14]{JP}
Tobias Johnson and Soumik Pal.
\newblock Cycles and eigenvalues of sequentially growing random regular graphs.
\newblock {\em Ann. Probab.}, 42(4):1396--1437, 2014.

\bibitem[LW08]{LoWaa}
Lisa Lorentzen and Haakon Waadeland.
\newblock {\em Continued fractions. {V}ol. 1}, volume~1 of {\em Atlantis
  Studies in Mathematics for Engineering and Science}.
\newblock Atlantis Press, Paris, second edition, 2008.

\bibitem[McK81]{McK}
Brendan~D. McKay.
\newblock Subgraphs of random graphs with specified degrees.
\newblock In {\em Proceedings of the {T}welfth {S}outheastern {C}onference on
  {C}ombinatorics, {G}raph {T}heory and {C}omputing, {V}ol. {II} ({B}aton
  {R}ouge, {L}a., 1981)}, volume~33, pages 213--223, 1981.

\bibitem[MH03]{MH}
J.~C. Mason and D.~C. Handscomb.
\newblock {\em Chebyshev polynomials}.
\newblock Chapman \& Hall/CRC, Boca Raton, FL, 2003.

\bibitem[MS03]{MS}
Hirobumi Mizuno and Iwao Sato.
\newblock The semicircle law for semiregular bipartite graphs.
\newblock {\em J. Combin. Theory Ser. A}, 101(2):174--190, 2003.

\bibitem[MWW04]{MWW}
Brendan~D. McKay, Nicholas~C. Wormald, and Beata Wysocka.
\newblock Short cycles in random regular graphs.
\newblock {\em Electron. J. Combin.}, 11(1):Research Paper 66, 12 pp.
  (electronic), 2004.

\bibitem[Sta99]{EC2}
Richard~P. Stanley.
\newblock {\em Enumerative combinatorics. {V}ol. 2}, volume~62 of {\em
  Cambridge Studies in Advanced Mathematics}.
\newblock Cambridge University Press, Cambridge, 1999.

\bibitem[TV10]{TV2}
Terence Tao and Van Vu.
\newblock Random matrices: universality of local eigenvalue statistics up to
  the edge.
\newblock {\em Comm. Math. Phys.}, 298(2):549--572, 2010.

\bibitem[TV11]{TaV}
Terence Tao and Van Vu.
\newblock Random matrices: universality of local eigenvalue statistics.
\newblock {\em Acta Math.}, 206(1):127--204, 2011.

\bibitem[TVW13]{TVW}
Linh~V. Tran, Van~H. Vu, and Ke~Wang.
\newblock Sparse random graphs: {E}igenvalues and eigenvectors.
\newblock {\em Random Structures Algorithms}, 42(1):110--134, 2013.

\end{thebibliography}

\end{document}